\documentclass[10 pt,reqno]{amsart}
\usepackage{color}
\usepackage{enumerate}
\usepackage{amssymb,amsmath,amsthm,latexsym,mathrsfs}
\usepackage{amsbsy,amsfonts,mathtools,slashed}
\usepackage{graphicx,color,yfonts}
\usepackage[numbers,sort]{natbib}
\usepackage{kotex}

\usepackage[colorlinks=true]{hyperref}
\hypersetup{linkcolor=red,citecolor=blue,filecolor=dullmagenta,urlcolor=blue} 

\newtheorem{thm}{Theorem}[section]
\newtheorem{lemma}[thm]{Lemma}
\newtheorem{prop}[thm]{Proposition}

\newtheorem{rem}[thm]{Remark}

\makeatother
\numberwithin{equation}{section}
\numberwithin{figure}{section}
\theoremstyle{plain}
\theoremstyle{plain}
\theoremstyle{plain}
\newcommand{\les}{\lesssim}
\newcommand{\lam}{{\lambda}}
\newcommand{\Lam}{{\Lambda}}

\newcommand{\vp}{{\varphi}}
\newcommand{\ve}{{\varepsilon}}
\newcommand{\de}{{\delta}}

\newcommand{\al}{{\alpha}}

\newcommand{\R}{{\mathbb R}}
\newcommand{\Z}{{\mathbb Z}}

\newcommand{\C}{{\mathbb C}}

\def\normo#1{\left\|#1\right\|}

\def\bra#1{\left\langle #1\right\rangle}

\def\wt#1{\widetilde{#1}}
\def\wh#1{\widehat{#1}}

\newcommand{\brad}{\langle D \rangle}

\newcommand{\thez}{{\theta_0}}
\newcommand{\theo}{{\theta_1}}
\newcommand{\thet}{{\theta_2}}
\newcommand{\theth}{{\theta_3}}

\newcommand{\thej}{{\theta_j}}

\begin{document}

	\title[Dirac equations with CSP gauge field]{Cauchy problem for Dirac equations with Chern-Simons-Proca
		gauge field}

\author{Hyungjin Huh}
\address{Department of Mathematics, Chung-Ang University, Seoul, 06974,  Republic of Korea}
\email{huh@cau.ac.kr}

\author{Kiyeon Lee}
\address{Stochastic Analysis and Application Research Center(SAARC), Korea Advanced Institute of Science and Technology, 291 Daehak-ro, Yuseong-gu, Daejeon, 34141, Republic of	Korea}
\email{kiyeonlee@kaist.ac.kr}

	
	\begin{abstract}
In this paper, we consider the Cauchy problem of  Dirac equations with Chern-Simons-Proca(CSP) gauge field. We investigate  global well-posedness and scattering theory for the solutions  with small initial data.  Main difficulties come from the fact that Strichartz estimate does not work and an absence of spinorial null structure which disturbs to show the global existence.  To overcome these obstacles, we employ the space-time resonance argument introduced by Germain-Masmoudi-Shatah \cite{gemash-2008,gemash-2012-jmpa,gemash-2012-annals}, as well as various resonance functions derived from Dirac operators and the normal form approach.  Our argument enables us to establish the global in time existence of solutions to Dirac equations with CSP gauge field. As a byproduct of our argument, we obtain the scattering results.
	\end{abstract}

		\maketitle

\tableofcontents

\section{Introduction}
We consider the Dirac equations with Chern-Simons-Proca(CSP) field:
\begin{align}\label{CSP-Dirac}
	\left\{\begin{aligned}
	&i \gamma^\mu \textbf{D}_\mu \psi = m \psi,\\
	 \frac12  & \epsilon^{\mu\nu\rho} F_{\mu\nu} + \lam A_\rho = -J^\rho,\\
	\end{aligned} \right.
\end{align}
where  the spinor field  is $\psi : \R^{1+2} \to \C^2$, the gauge fields are $A_\mu : \R^{1+2} \to \R$, and the covariant derivative is  $\textbf{D}_\mu = \partial_\mu -iA_\mu$.  The curvature is defined by $F_{\mu\nu} = \partial_\mu A_\nu - \partial_\nu A_\mu$. 
$\epsilon^{\mu\nu\rho}$ is the totally skew-symmetric tensor with $\epsilon^{012}=1$ and $\lam>0$ is a Proca coupling constant. The Dirac gamma matrices $\gamma^\mu$ are defined by
\begin{align*}
	\gamma^0 = \begin{pmatrix} 1 & 0 \\ 0 & -1	\end{pmatrix},\quad 	\gamma^1 = \begin{pmatrix} 0 & 1 \\ -1 & 0	\end{pmatrix}, \quad 
	\gamma^2 = \begin{pmatrix} 0 & -i \\ -i & 0	\end{pmatrix}.
\end{align*}
We also define $J^\mu = \bra{\psi, \, \al^\mu \psi}$ with $\al^\mu = \gamma^0 \gamma^\mu$ for $\mu =0, \,1, \,2$. The $\langle \cdot , \cdot \rangle$ denotes a standard complex inner product. Greek indices mean the space-time components $\mu,\nu = 0,\, 1,\,2$ and roman indices mean the spatial components $j=1,\, 2$ in the sequel. The mass parameter is $m \ge 0$. In this paper we only consider the massive case $m>0$.

Let us consider Lorenz gauge condition  $\partial^\mu A_\mu =0$. 
%\begin{align*}
%\partial^\mu A_\mu =0. \end{align*}
The static form of  the second equations in \eqref{CSP-Dirac} with above gauge condition can be described as follows:
\begin{align}\label{eq:static}
	\left\{\begin{aligned}
	& \partial_1 A_0 + \lam A_2 = J^2,\\
	& \partial_2 A_0 - \lam A_1 = -J^1,\\
	& F_{12} +  \lam A_0 = -J^0,\\
	& \partial_1 A_1 + \partial_2 A_2 = 0.
	\end{aligned} \right.
\end{align}
Combining Dirac part in \eqref{CSP-Dirac} and stationary gauge fields \eqref{eq:static}, we consider the following Cauchy problem  for  Dirac equations with Chern-Simons-Proca gauge fields:
\begin{align}\label{eq:CSP-Dirac}
	\left\{\begin{aligned}
i\gamma^\mu \textbf{D}_\mu \psi &= m \psi,\\
(\lam^2 - \Delta) A_1 &= \lam J^1 - \partial_2 J^0,\\
(\lam^2 - \Delta) A_2 &= \lam J^2 + \partial_1 J^0,\\
(\lam^2 - \Delta) A_0 &= - \lam J^0 + \partial_2 J^1 - \partial_1 J^2,
	\end{aligned} \right.
\end{align}
with initial data $\psi(0,x) = \psi_0(x)$.

 Chern-Simons-Proca model is a gauge theory that combines the Chern-Simons term and the Proca term, which is a model of the electromagnetic interaction. The CSP-Dirac system describes the weak interaction of the Dirac fermions and intermediate vector Boson. The $(1+2)$ dimensional Chern-Simons gauge theories give a useful description to fractional quantum Hall effect and anyonic superconductivity.  Especially,
the Abelian Chern-Simons theories \cite{Du, Hor} have  attracted much attention for both physicists and mathematicians.

The  Proca term is a classical way of giving a mass to gauge fields. From the mathematical point of view, the Proca term plays an important role in proving a non-trivial solution to the equations \eqref{eq:CSP-Dirac}. 
In fact, if the Proca term vanishes, $\lambda=0$ in \eqref{eq:CSP-Dirac}, then we have Laplace operator $-\Delta$  for gauge fields instead of Klein-Gordon operator $-\Delta + 1$.  Then the gauge fields are represented by     
\[
A = c\int_{\Bbb{R}^2} \frac{x_j-y_j}{ |x-y|^2} J (y) dy,
\]
which gives us difficulties in estimating $A_{\alpha}$. 
The   Proca term $(\lambda \neq 0)$ is important in our analysis.
Several mathematical papers \cite{A, Dr, H19, tsu2003} have dealt with Maxwell-Proca equations. The Chern-Simons-Proca system has been studied in \cite{dong2021, huh2010, Sa} (see also section 10.1 in \cite{book-aithey}).

In this paper, we are concerned with the global well-posedness and linear scattering theory for the system \eqref{eq:CSP-Dirac}.  Since we consider our Cauchy problem based on the small initial data, the signs of nonlinearity are not relevant. Then we normalize $\lam =1$ and $m=1$ for simplicity. Let us  reformulate \eqref{eq:CSP-Dirac} by multiplying $\beta:=\gamma^0$. 
Then we reduce our main equations \eqref{eq:CSP-Dirac} into
\begin{align}\label{eq:maineq}
		(-i\partial_t + \al \cdot D + m\beta)\psi = N_1(\psi,\psi,\psi) + N_2(\psi,\psi,\psi), 
\end{align}
where
\begin{align*}
	N_1(\psi,\phi,\vp) &= \frac1{1-\Delta} \bra{\phi,\vp}\psi -\frac1{1-\Delta}\bra{\phi,\al^1 \vp}\al^1\psi-\frac1{1-\Delta}\bra{\phi,\al^2 \vp}\al^2\psi,\\
N_2(\psi,\phi,\vp) &= \frac{\partial_1 \bra{\phi,\al^2\vp} - \partial_2 \bra{\phi,\al^1\vp}}{1 -\Delta} \psi + \frac{\partial_2}{1 -\Delta}\bra{\phi,\vp} \al^1 \psi  - \frac{\partial_1}{1 -\Delta}\bra{\phi,\vp} \al^2 \psi.
\end{align*}
Here, $\al = (\al^1,\al^2)$ with $\al^j = \gamma^0 \gamma^j$ and $D= -i\nabla$.

\begin{rem}
The equation \eqref{eq:maineq} obeys mass conservation law. If a solution $\psi$ is sufficiently smooth, then the mass $\|\psi(t)\|_{L_x^2}^2$ is conserved, that is, $\|\psi(t)\|_{L_x^2}^2 = \|\psi_0\|_{L_x^2}^2$ for all $t$ within an existence time interval. %See \cite{geosha} for the proof.
\end{rem}

There is a large amount of literature dealing with the problems of well-posedness and asymptotic behavior of small solutions of  the Dirac equations with Hartree type nonlinearity. To better put \eqref{eq:maineq} into context in relation with  the global well-posedness and scattering theory,  let us consider the following generalized Hartree type Dirac equations:
\begin{align}\label{eq:dirac-general}
	(-i\partial_t + \al \cdot D + m\beta )\psi = \left(V*\bra{\psi,\Gamma_1 \psi} \right) \Gamma_2\psi	\quad \mbox{on } \, \R^{1+d},   
\end{align}
where $V$ is potential and $\Gamma_1$, $\Gamma_2$ are either of $\gamma^\mu$ or suitable size identity matrix . We first mention  the Coulomb potential $V(x) = |x|^{-1}$, corresponding to $\lam =0$ case in \eqref{eq:CSP-Dirac}. In $d=3$, \cite{ckly2022,cloos} has proved the (modified) scattering results for the case $\Gamma_1 = \Gamma_2=  I_4$. They exploited the spinorial null structure \eqref{proj-commu} and space-time non-resonance  \cite{pusa}. For the generalized Coulomb potential $V(x) = |x|^{-a}(1<a<\infty)$ with $\Gamma_1 = \Gamma_2 =\beta$ and $d=2$, Cho, Ozawa and the second author have shown the linear scattering theory for a small solution with low regularity in \cite{chleoz2022}.  The Yukawa type potential $V=\bra{D}^{-2}$ is related with  our main equations.  In \cite{chle2021-die, tes2d}, the authors have studied the low regularity scattering results for Hartree-type Dirac equations in 2 dimension with small initial data in $L^2(\R^2)$ and  for $\Gamma_1 = \Gamma_2 = \beta$ by utilizing the bilinear estimates based on Strichartz estimates for Klein-Gordon equations.  We also refer to \cite{tes3d, yang}  for the same results for $H^\ve$ in $d=3$.

The nonlinear term of  Dirac equations with Yukawa potential is considerably similar to our nonlinearity $N_1(\psi,\psi,\psi)$ in \eqref{eq:maineq}. The significant difference is  that the nonlinearity  $\bra{D}^{-2}\bra{\psi,\gamma^0 \psi}\gamma^0 \psi$ in \cite{chle2021-die, tes2d} has  the spinorial null structure, which enables to perform the bilinear estimates, which are first introduced by Klainerman and Machedon \cite{klamach1993,klamach1996,klamach1997} for wave equations and developed by many mathematicians in \cite{behe3d,behe2d,ozarog2014, tes2d,tes3d} for Klein-Gordon type equations. In contrast to this fact,  our nonlinearity in \eqref{maineq-decom} has various interactions between the matrices $\Gamma_1$, $\Gamma_2$ included in  $\bra{D}^{-2}(|\psi|^2)\psi$. These interactions disturb to prove global well-posedness by making the bilinear estimates as well as Strichartz estimates not working well. This obstacle comes from the fact  that  we cannot eliminate the resonance case between two spinors without Dirac matrix $\beta$ even though we make use of  the decay effect, arising in potential $\bra{D}^{-2}$. 

In order to show the scattering phenomenon, we have to overcome these difficulties by employing other argumentation,  at the expense of regularities or spatial weight. We describe the main result of this paper.
\begin{thm}\label{thm:mainthm}
	Let $s \ge 50$. Suppose that the initial data $\psi_0$ satisfy the condition
	\begin{align}\label{eq:initial-condition}
		\|\psi_0\|_{H^s} + \|x \psi_0\|_{H^5} \le \ve_0,		
	\end{align}
	for sufficiently small $\ve_0 >0$. Then there exists a unique global solution $\psi \in C(\mathbb{R};H^s)$ to \eqref{eq:maineq}, which scatters in $H^5(\bra{x}^2 dx)$. 
\end{thm}

 Our argument has not been confined to the structure of nonlinearity. We can apply our argument for \eqref{eq:dirac-general} with arbitrary matrices $\Gamma_1$, $\Gamma_2 \in M_{2\times 2}(\C)$. For these reasons, it suffices to treat
\[
 \bra{D}^{-j}\bra{\psi,\al^\mu \psi} \al^\nu \psi,
\]
 for all $\mu$, $\nu=0,1, 2$ and $j =1, 2$. Especially, we fix $j =1$.

\begin{rem}
The system \eqref{CSP-Dirac} without Proca field $(\lam =0)$   was considered as Chern-Simons-Dirac(CSD)  equations. Since the singularity arising from potential  disrupts to verify the global well-posedness and scattering theory even if we impose  the static gauge fields condition, there is  only local in time solution. For the local well-posedness results for the CSD system, we refer to  \cite{boucanma2014, huh2007, huoh2016, pec2016,pec2019}. However, when $V= |x|^{-1}$ and $\Gamma_1 = \Gamma_2 = I_2$ in \eqref{eq:dirac-general}, we expect that the argument  in this paper can be  applied to show the global well-posedness and modified scattering phenomenon with suitable phase correction $($see \cite{kly2023}$)$.
\end{rem}

\subsection{Reformulation for Dirac equations} 
We use the representation of solution based on the massive Klein-Gordon equation. For the purpose of diagonalizing of \eqref{eq:maineq}, let us define the  projection operators $\Pi_\theta(D)$ by
$$
\Pi_{\theta}(D) := \frac12 \left( I +\theta \frac1{\brad}[\alpha \cdot D + \beta] \right),
$$
where $\left< D\right> := \mathcal F^{-1}\langle \xi \rangle\mathcal F$ and $\langle \xi \rangle := (1 + |\xi|^2)^\frac12$ for any $\xi \in \mathbb R^2$.
Then we get
\begin{align*}
	\al\cdot D + \beta = \brad (\Pi_{+}(D) - \Pi_{-}(D)),
\end{align*}
and
\begin{align}\label{proj-commu}
	\Pi_{\theta}(D)\Pi_{\theta}(D) = \Pi_{\theta}(D), \quad \Pi_{\theta}(D)\Pi_{-\theta}(D) = 0.
\end{align}
We denote $\Pi_\theta(D)\psi $ by $\psi_\theta$. Then the equation \eqref{eq:maineq} becomes the following system of semi-relativistic Hartree equations:
\begin{align}\label{maineq-decom}
	(-i\partial_t +\theta \brad)\psi_\theta = \Pi_{\theta}(D)[\bra{D}^{-1}\bra{\psi,\al^\mu \psi} \al^\nu \psi]
\end{align}
with initial data $\psi_{\theta}(0,\cdot) = \psi_{0, \theta} := \Pi_{\theta}(D)\psi_0$. The free solutions of \eqref{maineq-decom} are $e^{ - \theta it\brad}\psi_{0, \theta}$, where
$$
e^{-\theta it\langle D\rangle}f(x) = \mathcal F^{-1}( e^{- \theta it \langle \xi \rangle}\mathcal F f) = \frac1{(2\pi)^2}\int_{\mathbb R^2} e^{i(x\cdot \xi -\theta t\langle \xi \rangle)}\widehat f(\xi)\,d\xi.
$$
Here $\mathcal F$, $\mathcal F^{-1}$ are Fourier transform, its inverse, respectively. Note that $U(t)\psi_0$ is also the free solution of \eqref{eq:maineq},  where 
\[
U(t):= \Pi_+(D)e^{-it\bra{D}} + \Pi_-(D)e^{it\bra{D}}.
\]
Then, by Duhamel's principle, the Cauchy problem \eqref{maineq-decom} is equivalent to solve the integral equations:
\begin{align}\label{eq:inteq0}
	\begin{aligned}
		\psi_\theta(t) &= e^{-\theta it\langle D\rangle}\psi_{0, \theta}\\
		& \hspace{1cm}+  i \sum_{\substack{\thej \in \{\pm\}\\ j=1,2,3}}\int_0^t e^{-\theta i(t-t')\langle D\rangle}\Pi_{\theta}(D)\bra{D}^{-1}\left[\bra{\psi_\thet(t'),\al^\mu\psi_\theth(t')}\right]\al^\nu\psi_\theo(t')\,dt'.
	\end{aligned}
\end{align}

We call that the solution $\psi$ scatters forward (or backward) in Hilbert space  $\mathcal H$ if there exists $\psi^l \in C(\mathbb R; \mathcal H)$,  linear solutions to $(-i\partial_t + \alpha \cdot D + \beta)\psi^l = 0$, such that
\begin{align*}
	\|\psi(t) - \psi^l(t)\|_{\mathcal H} \to 0\;\;\mbox{ as }\;\;t \to +\infty, \quad (-\infty \,\, \mbox{respectively}).
\end{align*}
Equivalently, $\psi$ is said to scatter forward (or backward) in $\mathcal H$ if there exists  $\psi_{\theta }^l := e^{-\theta it\brad}\varphi_{\theta } \;(\varphi_{\theta} \in  \mathcal H)$ such that
\begin{align*}
	\|\psi_\theta(t) - \psi_{\theta}^l(t)\|_{\mathcal H} \to  0  \;\;\mbox{ as }\;\;t \to +\infty, \quad (-\infty \,\, \mbox{respectively}).
\end{align*}
By time reversal symmetry, it is enough to  consider the forward scattering. In view of  Theorem \ref{thm:mainthm}, it suffices to prove that, for $\theta \in \{+,-\}$, there exists $\phi_\theta$ such that
\begin{align}\label{aim:scattering}
	\normo{\psi_\theta(t) - e^{-\theta it\bra{D}}\phi_\theta }_{H^5(\bra{x}^2 dx)}  \to  0 \,\,\mbox{ as }\;\;t \to +\infty.
\end{align}

\subsection{Strategy of the proof} As we mentioned in known results,  we cannot utilize the bilinear and Strichartz estimates via $U^p$--$V^p$ spaces. For the purpose of obtaining the global well-posedness and showing the scattering phenomenon, we exploit the space-time resonance argument.  This argument was first introduced by Germain, Masmoudi, and Shatah in \cite{gemash-2008,gemash-2012-jmpa,gemash-2012-annals}. It observes a time or space oscillation in the oscillatory integral which nonlinear terms of the main equation make. More precisely, we set an a priori assumption, which incorporates sufficient  time decay to show scattering results at the expense of high regularity and weighted norm conditions. Indeed, we solve \eqref{eq:maineq} in a function space, given by the norm for some $\ve_1 >0$,
\begin{align*}
	\mathcal{S} := \left\{ \psi : \sum_{\theta \in \{\pm\}}\sup_{t \in [0,T]} \left( \|\psi_\theta(t)\|_{H^s} + \| x e^{\theta i t \bra{D}}\psi_\theta(t)\|_{H^5} \right) \le \ve_1\right\}.
\end{align*}
In view of dispersive estimates for Klein-Gordon operator, we need the second order weight condition to obtain global decay $t^{-1}$ in two dimensions. Since the decay effect of the nonlinearity is strong enough to show the linear scattering, we do not pursue the full decay here. However, we need to estimate decay of the weighted Sobolev space $W^{k,\infty}$ for some regularity $k>0$ that requires the initial data condition (see Proposition \ref{prop:timedecay}). This fact enables us to consider only the first-order weighted condition. Then, using the assumption and decay property, we perform the weighted energy estimates. For these estimates, we invoke the Duhamel's principle \eqref{eq:inteq0} with Fourier transform:
\begin{align}\label{eq:duhamel-res}
	\begin{aligned}
	e^{\thez it \bra{\xi}} \wh{ \psi_\thez} (t,\xi) &= \wh{\psi_{0,\thez}}(\xi) + i \sum_{\substack{\thej \in \{\pm\}\\ j=1,2,3}}\int_0^t \int \!\! \!\!\!\int_{\R^2 \times \R^2} e^{is\phi_\Theta(\xi,\eta,\sigma)} \bra{\eta}^{-1} \bra{\wh{f_\thet}(\sigma),\al^\mu \wh{f_\theth}(\eta+\sigma)} \\
	&\hspace{7.5cm}  \times \al^\nu \wh{f_\theo}(\xi-\eta) \,d\sigma d\eta ds,
	\end{aligned}
\end{align}
where a 4-tuple $\Theta = (\thez,\theo,\thet,\theth)$ and the resonant function 
\begin{align*}
\phi_\Theta(\xi,\eta,\sigma) =  \thez \bra{x} -\theo \bra{\xi-\eta}  -\thet \bra{\eta +\sigma} +\theth\bra{\sigma}.	
\end{align*}
 The main quantity that we shall handle is $x e^{\theta it\bra{D}} \psi_\theta$ in $L^2$ which corresponds to $(x - \theta it\bra{D}) \psi_\theta$ in $L^2$. This operator cannot be commute with the equation, however, we may bound the weighted norm as follows. We apply $\nabla_\xi$ to \eqref{eq:duhamel-res} via Fourier transform and then we obtain the worst term of the form
\begin{align}\label{eq:integ-res}
	\int_0^t s \iint_{\R^2 \times \R^2} \nabla_\xi \phi_\Theta(\xi,\eta,\sigma) e^{is\phi_\Theta(\xi,\eta,\sigma)} \bra{\eta}^{-1} \bra{\wh{f_\thet}(\sigma),\al^\mu \wh{f_\theth}(\eta+\sigma)} \al^\nu \wh{f_\theo}(\xi-\eta) \,d\sigma d\eta ds.
\end{align}
To bound \eqref{eq:integ-res}, we need to recover the time growth $s$. To this end, we exploit the space-time resonance argument. Indeed, we estimate \eqref{eq:integ-res} by separating the space resonant case ($\nabla_\eta \phi_\Theta =0$  or $\nabla_\sigma \phi_\Theta =0$), the time resonant case ($ \phi_\Theta =0$), and the space-time resonant case satisfying both resonant cases. The first main proof of our paper is as follows:
\begin{itemize}
	\item[(1)] The phase function $\phi_\Theta$ does not exhibit any non-resonant case when $\thez=\theo,\, \thet = \theth$. Fortunately, case $\thez=\theo$ has a strong null structure $\nabla_\xi \phi_\Theta(\xi,\eta,\sigma)\Big|_{\eta=0}=0$, sufficient to eliminate the resonant cases, i.e., it suffices only to handle the non-resonant case in terms of $\sigma$ oscillation via Lemma \ref{lem:esti-l2-decay} and we can establish to bound \eqref{eq:integ-res}. On the other hand, there is no such null structure when $\thez \neq \theo$ and however, the phase function is non-resonant in time, i.e., we have $|\phi_\Theta| \ge C$. Due to this time non-resonance, we may utilize the normal form approach. This approach leads us to treat the case in which the time derivatives fall on other spinors. To estimate these terms, we develop the bilinear estimates in which the time derivatives give the extra time decay effect (see Lemma \ref{lem:esti-timederi}). This lemma also enables us to recover the time loss $s$.  See the Section \ref{sec:res}.
	\end{itemize}

The rest of the main proof is devoted to considering the resonant case $\thet \neq \theth$. In contrast, since there is no oscillation with respect to $\sigma$ in this case, we cannot use the improved bilinear estimates (Lemma \ref{lem:esti-l2-decay}). Thus, other approaches to obtain an additional time decay effect have to be employed. We give the second main observation of this paper:
 \begin{itemize}
 	\item[(2)]  We decompose the case by considering the frequency support relations $|\xi|$, $|\xi-\eta|,\,|\eta+\sigma|$, and $|\sigma|$ to bound \eqref{eq:integ-res}. There are two approaches to obtain an extra time decay. One approach proceeds using the time non-resonant case. For the time resonant case, we observe the space resonance cases in which there is no oscillation in $\eta$ or $\sigma$,
\[
\nabla_\eta \phi_\Theta(\xi,\eta,\sigma) = 0 \quad \mbox{or}\quad  \nabla_\sigma \phi_\Theta(\xi,\eta,\sigma) = 0.
\]
In this observation, since the extra decay effect accompanies the singularities $|\eta|^{-1}$ or $|\sigma|^{-1}$, it is natural to employ the null structure which can eliminate these singularities. Among the sign relations, the case $\Theta = (+,-,-,+)$ does not give any structure to recover the singularities, even though we consider the multiplier $\nabla_\xi \phi_\Theta$. To avoid this difficulty, we add the $\nabla_\sigma \phi_\Theta$ into the multiplier  $\nabla_\xi \phi_\Theta$. Then this new multiplier gives rise to a sufficient null structure 
\[
\left[\nabla_\xi \phi_\Theta + \nabla_\sigma \phi_\Theta\right] (\xi,\eta,\sigma)\Big|_{\sigma=0}=0
\]
to reduce the singularity, arising from space resonance with respect to $\sigma$. Lastly, we estimate the remainder term with $\nabla_\sigma \phi_\Theta$ by taking advantage of $[\nabla_\sigma \phi_\Theta] e^{is\phi_\Theta} = s^{-1} \nabla_\sigma(e^{is\phi_\Theta})$ with integration by parts in $\sigma$. See the last case in  section \ref{sec:non-res} for the details. 
 \end{itemize}

\begin{rem}
This argument has been applied to show a global well-posedness and  asymptotic behavior of dispersive equations with various non-local differential operators including Dirac operator. We refer to \cite{ckly2022,   iopu2014, pusa, sauwan2021}.	
\end{rem}

\begin{rem}
	Regarding the global well-posedness and scattering results for the non-stationary Chern-Simons-Proca fields, we expect that we set appropriate a priori assumption by exploiting the space-time resonance argument. We treat this topic in the future work.
\end{rem}

\noindent \textbf{Organization.} The paper is organized as follows: In section \ref{sec:preli}, we provide a priori assumption for the bootstrap argument and  some preliminaries, in particular, the time decay effect  under  a priori assumption, Coifman-Meyer operator estimates and crucial lemmas by exploiting space-time resonance method. In section \ref{sec:3}, we prove Theorem \ref{thm:mainthm} and estimate the high Sobolev norm. We give the proof of Proposition \ref{prop:energy-weight} in section \ref{sec:mainproof}, where we also use the space-time resonance method together with the normal form approach by decomposing various resonance cases. 
\\

\noindent{\bf Notations.}
\ \\
$(1)$ $\bra{\cdot}$ denotes $(1+ |\cdot|^2)^{\frac12}$.\\
\ \\
$(2)$ (Mixed-normed spaces) For a Banach space $X$ and an interval $I$, $u \in L_I^q X$ iff $u(t) \in X$ for a.e. $t \in I$ and $\|u\|_{L_I^qX} := \|\|u(t)\|_X\|_{L_I^q} < \infty$. Especially, we abbreviate $L^p=L_x^p$ for the spatial norm and indicate the subscripts for only Fourier space norm. \\

\noindent$(3)$ (Sobolev spaces)  For $s \in \R$, $\|u\|_{H^s} := \|\bra{D}^s u\|_{L^2}$. Moreover, $\|u\|_{W^{s,p}} := \|\bra{D}^s u\|_{L^p}$ $p\neq 2$.\\

\noindent$(4)$ Different positive constants depending only on $a$, $b$ are denoted by the same letter $C$, if not specified. $A \lesssim B$ and 
$A\gtrsim B $ mean that $A \le CB$ and $A \geq  C^{-1}B$  respectively for some $C>0$. $A \sim B$  means  that $A \lesssim B$ and $A \gtrsim B$.\\

\noindent$(5)$ (Littlewood-Paley operators) Let $\rho$ be a Littlewood-Paley function such that $\rho \in C^\infty_0(B(0, 2))$ with $\rho(\xi) = 1$  for $|\xi|\le 1$ and  $\rho_{N}(\xi):= \rho\left(\frac{\xi}{N}\right) - \rho\left(\frac{2\xi}{N}\right)$ for $N \in 2^\mathbb{Z}$. Then we define the frequency projection $P_N$ by $\mathcal{F}(P_N f)(\xi) = \rho_{N}(\xi)\widehat{f}(\xi)$, and also $P_{\le N} := I - \sum_{N' > N}P_{N'}$. In addition $P_{N_1 \le \cdot \le N_2} := \sum_{N_1 \le N \le  N_2}P_k$. For $ N \in 2^\mathbb{Z}$ we denote $\widetilde{\rho_k} = \rho_{k-1} + \rho_k + \rho_{k+1}$. In particular, $\widetilde{P_N}P_N = P_N\widetilde{P_N} = P_N$, where $\widetilde{P_N} = \mathcal{F}^{-1}\widetilde{\rho_N}\mathcal{F}$. Especially we denote $P_N f $ by $f_{N}$ for any measurable function $f$ and  we do not distinguish $P_N$ and $\wt{P_N}$ throughout this paper.\\

\noindent$(6)$ Let $\textbf{A}=(A_i)$, $\textbf{B}=(B_i) \in \R^n$. Then $\textbf{A} \otimes \textbf{B}$ denotes the usual tensor product such that $(\textbf{A} \otimes \textbf{B})_{ij} = A_iB_j$. We also denote a tensor product of $\textbf{A} \in \C^n$ and $\textbf{B} \in \C^m$  by a matrix $\textbf{A} \otimes \textbf{B} = (A_iB_j)_{\substack{i=1,\cdots,n\\i=1,\cdots,m}}$. We often denote $\textbf{A} \otimes \textbf{B}$ simply by  $\textbf{A}\textbf{B}$. We also consider the product of  $x \in \R^3$ and $f \in \C^4$ by $xf = x \otimes f$. Moreover, we denote $\nabla f := (\partial_j f_i)_{ij}$.

%%%%%%%%%%%%%%%%%%%%%%%%%%%%%%%%%%%%%%%%%%%%%%%%%%%%%%%%%%%%%%%%%%%%%%%%%%%%%%%%%%%%%%%%%%%%%%%%%%%%%%%%%%%%%%%%%%%%%%%%%%%%%%%%%%%%%%%%%%%%%%%%%%%%%%%%%%%%%%%%%%%%%%%%%%%%%%%%%%%%%%%%%%%%%%%%%%%%%%%%%%%%%%%%%%%%%%%%%%%%%%%%%%%%%%%%%%%%%%%%%%%%%%%%%%%%%%%%%%%%%%%%%%%%%%%%%%%%%%%%%%%%%%%%%%%%%%%%%%%%%%%%%%%%%%%%%%%%%%%%%%%%%%%%%%%%%%%%%%%%%%%%%%%%%%%%%%%%%%%%%%%%%%%%%%%%%%%%%%%%%%%%%%%%%%%%%%%%%%%%%%%%%%%%%%%%%%%%%%%%%%%%%%%%%%%%%%%%%

\section{Preliminaries}\label{sec:preli}

\subsection{A priori assumption} 
As we mentioned in introduction, our proof of main theorem is based on the bootstrap argument. To perform that, we give a priori assumption   as follows: For $\ve_1 >0$,
\begin{align}\label{eq:apriori}
	\|\psi\|_{\Sigma_T} :=	\|\psi_+\|_{\Sigma_T^+}  + \|\psi_-\|_{\Sigma_T^-} \les \ve_1,
\end{align}
where, for $\theta \in \{+,-\}$,
\begin{align*}
	\|\psi_\theta\|_{\Sigma_T^\theta} := \sup_{t \in [0,T]} \left(
	\|\psi_\theta(t)\|_{H^s} + \|x e^{\theta it\bra{D}} \psi_\theta(t)\|_{H^5}  \right).
\end{align*}

Under this assumption, we have the time decay effect.
\begin{prop}\label{prop:timedecay}
	Assume that $\psi$ satisfies a priori assumption \eqref{eq:apriori} for given $\ve_1$ and $T$. Then we have a  decay
	\begin{align}\label{eq:timedecay}
		\|\psi(t)\|_{W^{k,\infty}} \les \bra{t}^{-\frac34},
	\end{align}
	for any $t \in [0,T]$ and $0 \le k \le 7$.
\end{prop}

\begin{rem}
	Regarding the time decay for linear Klein-Gordon equations, we may have the decay $t^{-1}$ for sufficient smooth function $($see \cite{chozxi}$)$. It suffices to obtain the decay $t^{-\frac12 -\ve}$ to handle the cubic nonlinearity. In \eqref{eq:timedecay}, $\frac34$ is not optimal and we may reduce the regularity assumption in Theorem \ref{thm:mainthm} by optimizing the time decay. In this paper, we do not pursue optimality of the regularity $H^s(s \ge 50)$ or $H^5(\bra{x}^2 dx)$ in \eqref{eq:apriori}. Our regularity condition is necessary to obtain $W^{k,\infty}$ decay with the regularity condition $0 \le k \le 7$. 
\end{rem}

\begin{proof} Note that we write
	\[
	\psi_{\theta}(t,x) = \frac1{(2\pi)^2}\int_{\R^2}e^{it\phi_{\theta}(\xi)}\widehat{f_{\theta}}(\xi)d\xi,
	\]
	where
	\begin{align}
		\phi_{\theta}(\xi):={-\theta}\bra{\xi}+\xi\cdot\frac{x}{t}.\label{theta-phase}
	\end{align}
	If $t \les 1$, then by H\"older's inequality we have
	\begin{align*}
		|\bra{D}^k\psi_{\theta}(t,x)| &\les  \|\bra{\xi}^{k-s}\|_{L^2}\|\psi_\theta\|_{H^s} \les \|\psi\|_{\Sigma_{T}}
	\end{align*}
	for $0 \le k \le 7$. Hence it suffices to show that
	\begin{align}\label{eq:claim-decay}
		\begin{aligned} &\|\psi_{\theta}(t)\|_{W^{k,\infty}} \les  t^{-\frac34} \ve_1, 
		\end{aligned}
	\end{align}
	for $t \gg 1$.	For the purpose of showing \eqref{eq:claim-decay}, we take a frequency decomposition of $\psi_{\theta}$ such that
	\[
	\psi_{\theta}(t,x)=\sum_{N\in2^{\mathbb{Z}}}\int_{\R^2} e^{it\phi_{\theta}(\xi)}\rho_{N}(\xi) \widehat{f_{\theta}}(\xi)d\xi.
	\]
	Then we get
	\begin{align*}
		t^{\frac34}|\bra{D}^k\psi_{\theta}(t,x)|\le\sum_{N\in2^{\mathbb{Z}}}I_{N}(t,x),
	\end{align*}
	where
	\[
	I_{N}(t,x)= t^{\frac34}\left|\int_{\mathbb{R}^{2}}e^{it\phi_{\theta}(\xi)}\rho_{N}(\xi)\bra{\xi}^k\widehat{f_{\theta}}(\xi)d\xi\right|.
	\]
	The proof is based on the stationary phase method. For this, we decompose the frequency support as
	\[
	\sum_{N\in2^{\mathbb{Z}}}I_{N}(t,x)=\left(\sum_{N\le t^{-\frac12}}+\sum_{N\ge t^{\frac1s}}+\sum_{t^{-\frac{1}{2}}\le N\le t^{\frac1s}}\right)I_{N}(t,x).
	\]
	The low-frequency part can be estimated as
	\[ \sum_{N\le t^{-\frac12}}I_{N}(t,x)\les\sum_{N\le t^{-\frac12}}t^{\frac34}\|\rho_{N}\|_{L_{\xi}^{\frac43}}\normo{\widehat{f_{\theta}}}_{L_{\xi}^{4}} \les \|\bra{x} f_\theta\|_{L^2} \les \ve_1.
	\]
	For the high-frequency part, we have
	\begin{align*}
		\sum_{N\ge t^{\frac1s}}I_{N}(t,x) & \les\sum_{N\ge t^{\frac1s}}t^{\frac34} N^{1+k}\|\rho_{N}\widehat{f_{\theta}}\|_{L_{x}^{2}} \les \sum_{N\ge t^{\frac1s}}t^{\frac34} N^{8}N^{-s}\|f_{\theta}\|_{H^{s}}\\
		&\les  t^\frac34 t^{(8-s)\frac1s } \|f_{\theta}\|_{H^{s}}\les \ve_1.
	\end{align*}

	Let us now focus on the mid-frequency part. To prove this part, we use both the non-stationary and stationary
	phase methods. 	From \eqref{theta-phase}, we have
	\[
	\nabla_{\xi}\phi_{\theta}(\xi)=-\theta\frac{\xi}{\bra{\xi}}+\frac{x}{t}.
	\]
	Then, when $|x|\ge t$, the phase $\phi_{\theta}$ is non-stationary. Indeed, we have
	\begin{align}
		|\nabla_{\xi}\phi_{\theta}(\xi)|\gtrsim\left|\frac{|x|}{t}-\frac{|\xi|}{\bra{\xi}}\right|\gtrsim1-\frac{|\xi|}{\bra{\xi}}\gtrsim \bra{\xi}^{-2}.\label{non-stat-1}
	\end{align}
	On the other hand, the phase $\phi_{\theta}$ could be stationary around
	$\xi_{0}$ when $|x|<t$:
	\[
	\nabla_{\xi}\phi_{\theta}(\xi_{0})=0\quad \mbox{where}\quad \xi_{0}=-\theta\frac{x}{\sqrt{t^{2}-|x|^{2}}}.
	\]
	We now set $|\xi_{0}|\sim N_{0} \in 2^\Z$. If $N\nsim N_{0}$, then the phase is non-stationary and simple calculation yields that
	\begin{align}
		\Big|\nabla_{\xi}\phi_{\theta}(\xi)\Big|\gtrsim\max\left(\frac{|\xi-\xi_{0}|}{\bra{N}^{3}},\frac{|\xi-\xi_{0}|}{\bra{N_{0}}^{3}}\right).\label{non-stat-2}
	\end{align}
	Let $\mathbf p_\phi = \frac{\nabla_\xi \phi_\theta}{|\nabla_\xi \phi_\theta|^2}$. Then we write the mid-frequency of
	$I_{N}(t,x)$ by integration by parts as follows:
	\[
	t^{\frac34}\int_{\R^2}e^{it\phi_{\theta}(\xi)}\rho_{N}(\xi) \bra{\xi}^k\widehat{f_{\theta}}(\xi)d\xi = \sum_{j = 1}^2I_{N}^{j}(t,x),
	\]
	where
	\begin{align}
		\begin{aligned}\label{eq-i123}
			I_{N}^{1}(t,x) & =it^{-\frac14}\int_{\R^2}e^{it\phi_{\theta}(\xi)}\mathbf p_\phi \nabla_\xi\left(\bra{\xi}^k\rho_N\widehat{f_{\theta}}(\xi)\right)\, d\xi,\\
			I_{N}^{2}(t,x) & =it^{-\frac14}\int_{\R^2}e^{it\phi_{\theta}(\xi)} (\nabla_\xi \mathbf p_\phi)  \left(\bra{\xi}^k\rho_N\widehat{f_{\theta}}(\xi)\right)\, d\xi.
		\end{aligned}
	\end{align}
	By \eqref{non-stat-1} and \eqref{non-stat-2}, we obtain the following bounds independent
	of $N_{0}$:
	\begin{align}
		\begin{aligned}\left|\mathbf p_\phi\right| & \les\bra{N}^{3}N^{-1},\\
			\left|\nabla_{\xi}\mathbf p_\phi\right| & \les\bra{N}^5N^{-2}.
		\end{aligned}
		\label{bound-phase}
	\end{align}
	Using \eqref{bound-phase} and Sobolev embedding $H^1 \hookrightarrow L^p$ ($2 \le p < \infty$), we see that
	\begin{align*}
		\Big|I_{N}^{1}(t,x)\Big| & \les t^{-\frac14}\bra{N}^{3}N^{-1}\left\Vert \nabla_{\xi}\left([\Lam(\xi)]^k\rho_{N}\widehat{f_{\theta}}\right)\right\Vert _{L_{\xi}^{1}}\\
		& \les t^{-\frac14}\bra{N}^{3+k}N^{-1}\left(N\|xf_{\theta}\|_{L^{2}} + N^{1-2\ve}\normo{\wh{f_{\theta}}}_{L_{\xi}^\frac1{\ve}}\right),
	\end{align*}
	which implies that
	\begin{align*}
		\sum_{t^{-\frac{1}{2}}\le N\le t^{\frac1s}}\Big|I_{N}^{1}(t,x)\Big| &\les \sum_{t^{-\frac{1}{2}}\le N \le 1}\Big|I_{N}^{1}(t,x)\Big| + \sum_{1 \le N\le t^{\frac1s}}\Big|I_{N}^{1}(t,x)\Big|\\
		&\les  t^{-\frac14+ \ve}  \|\bra{x}f_{\theta}\|_{L^2}  + t^{-\frac14 +\frac8s}\|\bra{x}f_{\theta}\|_{L^{2}} \les \ve_1.
	\end{align*}
In a similar way to above, we also obtain
	\begin{align*}
		|I_{N}^{2}(t,x)|  & \les t^{-\frac14}\bra{N}^{5}N^{-2}\normo{\bra{\xi}^k\rho_{N}\widehat{f_{\theta}}}_{L_{\xi}^{1}}\\
		& \les t^{-\frac14}\bra{N}^{5+k}N^{-\frac12}\|\bra{x}f_\theta\|_{L^2},
	\end{align*}
	which leads us to
	\[
	\sum_{t^{-\frac{1}{2}}\le N\le t^{\frac1s}}|I_{N}^{2}(t,x)| \les \ve_1.
	\]
	
	It remains to consider the stationary phase case, $N\sim N_{0}$. We
	further decompose $|\xi-\xi_{0}|$ into dyadic pieces $L$. Let  $L_{0}\in2^{\mathbb{Z}}$
	satisfy that $\frac{L_{0}}{2}<t^{-\frac{1}{2}}\le L_{0}$. Then we write
	\[
	t^\frac34\left|\int_{\R^2}e^{it\phi_{\theta}(\xi)}\rho_{N}(\xi)  \bra{\xi}^k\widehat{f_{\theta}}(\xi)d\xi\right|\le\sum_{L=L_{0}}^{2^{10}N}|J_{L}|,
	\]
	where
	\begin{align*}
		J_{L}(t,x):= \left\{\begin{aligned} & t^\frac34\int_{\mathbb{R}^{3}}e^{it\phi_{\theta}(\xi)}\rho_{\le L_{0}}(\xi-\xi_{0})\rho_{N}(\xi) \bra{\xi}^k\widehat{f_{\theta}}(\xi)d\xi\qquad\mbox{when }L=L_{0},\\
			& t^\frac34\int_{\mathbb{R}^{3}}e^{it\phi_{\theta}(\xi)} \rho_{L}(\xi-\xi_{0})\rho_{N}(\xi) \bra{\xi}^k\widehat{f_{\theta}}(\xi)d\xi\quad\qquad\mbox{when }L>L_{0}.
		\end{aligned}\right.
	\end{align*}
	The case $L=L_{0}$ is the stationary one. The fact $L_{0}\sim t^{-\frac{1}{2}}$ and Sobolev embedding  imply that 
	\[
	|J_{L_{0}}|\les t^\frac34 L_{0}^{\frac32}\|\rho_{N}\widehat{f_{\theta}}\|_{L_{\xi}^{4}}\les \ve_1.
	\]
	The case $L>L_{0}$ returns to non-stationary phase. By integration by parts, we write $J_{L}(t,x)$ as in \eqref{eq-i123}:
	\[
	J_{L}(t,x) = \sum_{j = 1}^2 J_{L}^{j}(t,x),
	\]
	where
	\begin{align*}
		\begin{aligned}J_{L}^{1}(t,x) & =it^{-\frac14}\int_{\R^2}e^{it\phi_{\theta}(\xi)}\mathbf p_\phi \nabla_\xi\left(\bra{\xi}^k\widehat{f_{\theta}}(\xi)\rho_{N}(\xi)\rho_{L}(\xi-\xi_{0})\right)d\xi,\\
			J_{L}^{2}(t,x) &
			=it^{-\frac14}\int_{\R^2}e^{it\phi_{\theta}(\xi)}(\nabla_\xi \mathbf p_\phi) \left(\bra{\xi}^k\widehat{f_{\theta}}(\xi)\rho_{N}(\xi)\rho_{L}(\xi-\xi_{0})\right) d\xi.
		\end{aligned}
	\end{align*}
	Each $J_{L}^{j}$ can be estimated similarly to $I_N^j$ with the following phase
	bounds:
	\begin{align*}
		\begin{aligned}\left|\mathbf p_\phi\right| & \les \bra{N}^{3}L^{-1}, \\
			\left|\nabla_\xi \mathbf p_\phi\right| & \les \bra{N}^{5}L^{-2}.
		\end{aligned}
	\end{align*}
	We omit  details of the proof. This finishes the proof of time decay \eqref{eq:timedecay}.
\end{proof}

\subsection{Useful estimates}
Here we give several lemmas, useful in the sequel. We begin with introducing a multiplier lemma about pseudo-product operators, which will be repeatedly used in our main proof.
\begin{lemma}[Coifman-Meyer operator estimates]  \label{lem:coifmey}   
	Assume that a multiplier $\textbf{m}$ satisfies 
	\begin{align*}
			\|\textbf{m}_1\|_{\rm CM} &:= \left\Vert \iint_{\R^2 \times \R^2}\mathbf{m}_1(\xi,\eta)e^{ix\cdot\xi}e^{iy\cdot\eta}d\xi d\eta\right\Vert _{L_{x,y}^{1}(\R^{2+2})} < \infty,\\
			\|\textbf{m}_2\|_{\rm CM} &:= \left\Vert \int \!\!\!\!\!\int\!\!\!\!\!\int_{\R^2 \times \R^2 \times \R^2}\mathbf{m}_2(\xi,\eta,\sigma)e^{ix\cdot\xi}e^{iy\cdot\eta}e^{iz\cdot\sigma}d\xi d\eta\right\Vert _{L_{x,y,z}^{1}(\R^{2+2+2})} < \infty.
	\end{align*}
	Then, for $\frac{1}{p}+\frac{1}{q}=\frac{1}{2}$,
	\begin{align*}
		\left\Vert \int_{\mathbb{R}^{2}}\mathbf{m}_1(\xi,\eta)\widehat{\psi}(\xi\pm\eta)\widehat{\phi}(\eta)d\eta\right\Vert _{L_{\xi}^{2}}\les \|\textbf{m}_1\|_{\rm CM}\|\psi\|_{L^{p}}\|\phi\|_{L^{q}}.
	\end{align*}
Moreover, for $\frac{1}{p}+\frac{1}{q} + \frac1r=\frac{1}{2}$,
	\begin{align*}
	\left\Vert \iint_{\mathbb{R}^{2} \times \R^2}\mathbf{m}_2(\xi,\eta,\sigma)\widehat{\psi}(\xi-\eta)\widehat{\phi}(\eta \pm \sigma) \wh{\vp}(\sigma) \,d\sigma d\eta\right\Vert _{L_{\xi}^{2}}\les \|\textbf{m}_2\|_{\rm CM}\|\psi\|_{L^{p}}\|\phi\|_{L^{q}}\|\vp\|_{L^{r}}.
\end{align*}
%	and for $\frac{1}{p}+\frac{1}{q}+\frac{1}{r} =1$,
%	\begin{align*}
%		\left| \iint_{\mathbb{R}^{3+3}}\mathbf{m}(\eta,\sigma)\widehat{\psi_1}(\eta\pm\sigma)\widehat{\psi_2}(\eta)\wh{\psi_3}(\sigma) d\eta  d\sigma \right| \les \|\textbf{m}\|_{\rm CM}\|\psi_1\|_{L^{p}}\|\psi_2\|_{L^{q}}\|\psi_3\|_{L^{r}}.
%	\end{align*}
\end{lemma}

These multiplier estimates lead us the following lemma by exploiting the space resonance approach.
\begin{lemma}[Non-resonant estimates]\label{lem:esti-l2-decay}
	Assume that $\psi$ satisfies a priori assumption \eqref{eq:apriori}.   Then we have
	\begin{align}
					%\|P_N \bra{\psi_\theo(t),\al^\mu\psi_\thet(t)}\|_{L_x^2} &\les  s^{-\frac34}\ve_1^2, \label{eq:esti-l2}\\
		\|P_N \bra{\psi_\theta(t),\al^\mu\psi_\theta(t)}\|_{L_x^2} &\les N^{-1} \bra{t}^{-\frac74+\gamma}\ve_1^2 \label{eq:esti-l2-decay}
	\end{align}
for sufficiently small $0 <\gamma \ll 1$ and $\theta \in \{ +,-\}$.
\end{lemma}
\begin{proof}
	% By H\"older inequality and Proposition \ref{prop:timedecay}, we have \eqref{eq:esti-l2} straightforwardly.
Taking Fourier transform in the left-hand side of \eqref{eq:esti-l2-decay}, we have 
\begin{align}\label{eq:inte1}
	\mathcal F \left[P_N \bra{\psi_\theta(t), \al^\mu \psi_\theta(t)} \right] (\xi)= \rho_N(\xi) \int_{\R^2} e^{is\varphi_{\theta}(\xi,\eta)} \bra{ \wh{f_\theta}(\eta), \al^\mu\wh{f_\theta}(\xi + \eta)} \,d\eta ,
\end{align}
where
\begin{align}\label{eq:phase-theta}
	\varphi_\theta(\xi,\eta) = \theta( \bra{\eta} -  \bra{\xi + \eta}).
\end{align}
To extract an extra time decay in \eqref{eq:inte1}, we use the space resonance approach. To this end, we note that
\begin{align}\label{eq:space-res-1}
	\left|\nabla_\eta\vp_{\theta}(\xi,\eta) \right| = \left| \frac{\eta}{\bra{\eta}} -  \frac{\xi+\eta}{\bra{\xi+\eta}} \right| \gtrsim \begin{cases}
		\frac{N}{\bra{N_1}^3}             & \mbox{ for } N_1 \sim N_2,\\
		\frac{N_1}{\bra{N_1}} 	 & \mbox{ for } N_1 \gg N_2,
	\end{cases}
\end{align}
where $|\eta| \sim N_1$ and $|\xi+\eta| \sim N_2$ for the dyadic numbers $N_1$, $N_2 \in 2^\Z$. Here we may only consider that $N_1 \gtrsim N_2$ due to a symmetry.  Using a relation 
\[
e^{is \vp_{\theta}} = -i \frac1s \frac{\nabla_\eta \varphi_{\theta} \cdot \nabla_\eta e^{is\vp_{\theta}}}{|\nabla_\eta \varphi_{\theta}|^2},
		\]
 the integration by parts in $\sigma$ leads us that \eqref{eq:inte1} is bounded by the following terms:
\begin{subequations}
	\begin{align}
		&  s^{-1}\rho_N(\xi) \int_{\R^2} e^{is\varphi_{\theta}(\xi,\eta)} \nabla_\eta m_{\theta,\textbf{N}}(\xi,\eta)\bra{ \wh{f_{\theta,N_1}}(\eta), \al^\mu\wh{f_{\theta,N_2}}(\xi + \eta)} \,d\eta, \label{eq:biliear1}\\
		& s^{-1}\rho_N(\xi) \int_{\R^2} e^{is\varphi_{\theta}(\xi,\eta)} m_{\theta,\textbf{N}}(\xi,\eta)\bra{\rho_{N_1}(\eta) \nabla_\eta\wh{f_\theta}(\eta),\al^\mu \wh{f_{\theta,N_2}}(\xi + \eta)} \,d\eta,\label{eq:biliear2}\\
	& s^{-1}\rho_N(\xi) \int_{\R^2} e^{is\varphi_{\theta}(\xi,\eta)} m_{\theta,\textbf{N}}(\xi,\eta)\bra{ \wh{f_{\theta,N_1}}(\eta),\al^\mu \rho_{N_2}(\xi+\eta)\nabla_\eta\wh{f_{\theta}}(\xi + \eta)} \,d\eta,\label{eq:biliear3}
	\end{align}
\end{subequations}
plus a similar term, where
\[
m_{\theta,\textbf{N}}(\xi,\eta) = \frac{\nabla_\eta \varphi_{\theta}(\xi,\eta)}{|\nabla_\eta \varphi_{\theta}(\xi,\eta)|^2}\rho_{N}(\xi)\rho_{N_1}(\eta)\rho_{N_2}(\xi+\eta).
\]
By Lemma \ref{lem:coifmey} with 
	\begin{align*}
	\|\nabla_\eta^\ell m_{\theta,\textbf{N}}\|_{\rm CM} \les \begin{cases}
		N^{-1}N_1^{-\ell} \bra{N_1}^{11+2\ell}            & \mbox{ for } N_1 \sim N_2,\\
		N_1^{-1-\ell} \bra{N_1} 	 & \mbox{ for } N_1 \gg N_2,
	\end{cases}
\end{align*}
for $\ell =0$, $1$,  we obtain 
\begin{align*}
	\sum_{N_1 \sim N_2}\|\eqref{eq:biliear1}\|_{L_\xi^2} \les \sum_{N_1 \sim N_2} s^{-1}N^{-1} N_1^{-1}\bra{N_1}^{13} \|\psi_{\theta,N_1}\|_{L^2} \|\psi_{\theta,N_2}\|_{L^{\infty}} \les N^{-1} s^{-\frac74}  \ve_1^2,
\end{align*}
and
\begin{align*}
	\sum_{N_1 \gg N_2}\|\eqref{eq:biliear1}\|_{L_\xi^2} \les 	\sum_{N_1 \gg N_2} s^{-1}N_1^{-2}  \|\psi_{\theta,N_1}\|_{L^\infty} \|\psi_{\theta,N_2}\|_{L^2} \les N^{-1} s^{-\frac74}  \ve_1^2.
\end{align*}	
Similarly, we have
\begin{align*}
		\sum_{N_1 \sim N_2}\|\eqref{eq:biliear2}\|_{L_\xi^2} &\les \sum_{N_1 \sim N_2}s^{-1}N^{-1} \bra{N_1}^{11} \|P_{N_1}(xf_\theta)\|_{L^2} \|\psi_{\theta,N_2}\|_{L^{\infty}} \les N^{-1} s^{-\frac74}  \ve_1^2,\\
		\sum_{N_1 \gg N_2}\|\eqref{eq:biliear2}\|_{L_\xi^2} &\les \sum_{N_1 \gg N_2} s^{-1}N^{-1}\bra{N} \|P_{N_1}(xf_\theta)\|_{L^2} \|\psi_{\theta,N_2}\|_{L^{\infty}} \les N^{-1}  s^{-\frac74+\gamma}  \ve_1^2.
\end{align*}
Estimates for \eqref{eq:biliear3} can be handled similarly. This finishes the proof of \eqref{eq:esti-l2-decay}.
		\end{proof}	
	
\begin{lemma}\label{lem:esti-timederi}
	Assume that $\psi$ satisfies a priori assumption \eqref{eq:apriori}. Then we have
	\begin{align}
		& \|\partial_t P_Nf_{\theta}(t)\|_{L^2} \les \bra{N}^{-5}\bra{t}^{-\frac32} \ve_1,	\label{eq:esti-timederi-f}\\
		& \|P_N \partial_t \bra{\psi_\theta(t), \al^\mu \psi_\theta(t)}\|_{L^2} \les \bra{t}^{-\frac32+\frac{\gamma}4} \ve_1^2, \label{eq:esti-timederi-bi}
	\end{align}
for sufficiently small $0 < \gamma \ll 1$ and $\theta \in \{+,-\}$.
\end{lemma}	
\begin{proof}
	By Duhamel's principle, recall that
	\[
	\partial_tf_{\theta}(t) = C \Pi_\theta(D)\sum_{\thej \in \{\pm\},j=1,2,3}\left[\bra{D}^{-1}\bra{\psi_\thet(t),\al^\mu\psi_\theth(t)}\right] \al^\nu\psi_\theo(t).
	\]
Then we have
\begin{align*}
		\|\partial_t P_Nf_{\theta}(t)\|_{L^2} &\les \bra{N}^{-5}\sum_{\substack{\thej \in \{\pm\},j=1,2,3}} \| \bra{D}^{-1} \bra{\psi_\thet(t),\al^\mu \psi_{\theth}(t)} \|_{H^5}  \normo{ \psi_\theo}_{W^{5,\infty}}\\
		&\les  \bra{N}^{-5} \bra{t}^{-\frac32}\ve_1^3.
\end{align*}
This finishes the proof of \eqref{eq:esti-timederi-f}.

Let us consider \eqref{eq:esti-timederi-bi}. Since the time derivative falls on the phase function or spinors, the left-hand side of \eqref{eq:esti-timederi-bi} is bounded by the $L^2$-norm of the following terms: 
\begin{subequations}
	\begin{align}
	&\rho_N(\xi) \int_{\R^2}  \vp_{\theta}(\xi,\eta)e^{it\vp_{\theta}(\xi,\eta)}  \bra{\wh{f_\theta}(t,\eta), \al^\mu \wh{f_\theta}(t,\xi + \eta)} \,d\eta,\label{esti-timederi-1}\\
	&\rho_N(\xi) \int_{\R^2} e^{it\vp_{\theta}(\xi,\eta)}  \bra{\partial_t\wh{f_\theta}(t,\eta), \al^\mu \wh{f_\theta}(t,\xi+\eta)} \,d\eta,\label{esti-timederi-2}
	\end{align}
\end{subequations}
and symmetric term where  time derivative $\partial_t$ falls on the other spinor. Here $\vp_\theta$ is defined in \eqref{eq:phase-theta}. To estimate \eqref{esti-timederi-1} in $L^2$, we first localize $|\eta|$, $|\xi +\eta|$ into the dyadic numbers $N_1,N_2 \in 2^\Z$ and we observe the phase function $\vp_{\theta}$: by the symmetry, we may assume that $N_1 \gtrsim N_2$. Then we have
\[
\left|\vp_{\theta}(\xi,\eta) \right| = \left| \bra{\eta} -  \bra{\xi+\eta} \right| \les N
\]
and
\begin{align}\label{eq:space-res-2}
\left|\nabla_\eta\vp_{\theta}(\xi,\eta) \right| = \left| \frac{\eta}{\bra{\eta}} -  \frac{\xi+\eta}{\bra{\xi+\eta}} \right| \gtrsim \begin{cases} 	\frac{N}{\bra{N_1}^3}             & \mbox{ for } N_1 \sim N_2,\\
	\frac{N_1}{\bra{N_1}} 	 & \mbox{ for } N_1 \gg N_2.
\end{cases}
\end{align}
%On the other hand, for $N_2 \ll N_1$, we also have
%\[
%\left|\vp_{12}(\xi,\eta) \right| = \left|\theo \bra{\eta} - \thet \bra{\xi+\eta} \right| \sim \bra{N_1}
%\]
%and
%\begin{align}\label{eq:space-res-3}
%\left|\nabla_\eta\vp_{12}(\xi,\eta) \right| = \left|\theo \frac{\eta}{\bra{\eta}} - \thet \frac{\xi+\eta}{\bra{\xi+\eta}} \right| \gtrsim 
%	\frac{N_1}{\bra{N_1}}.
%\end{align}
From this observation, one can see that
\[
e^{it\vp_{\theta}(\xi,\eta)} =-\frac{i}{t} \frac{\nabla_\eta \vp_{\theta}(\xi,\eta) \cdot \nabla_\eta e^{it\vp_{\theta}(\xi,\eta)}}{\left|\nabla_\eta\vp_{\theta}(\xi,\eta) \right|^2}.
\]
Using the above relation, we perform the normal form transform to bound \eqref{esti-timederi-1} by the following terms:
\begin{subequations}
\begin{align}
	&\frac{\rho_N(\xi)}{t} \int_{\R^2} e^{it\vp_{\theta}(\xi,\eta)} \nabla_\eta m_{\textbf{N}}(\xi,\eta)  \bra{\wh{f_{\theta,N_1}}(t,\eta), \al^\mu \wh{f_{\theta,N_2}}(t,\xi + \eta)} \,d\eta,\label{esti-timederi-1-1}\\
	&\frac{\rho_N(\xi)}{t} \int_{\R^2} e^{it\vp_{\theta}(\xi,\eta)} m_{\textbf{N}}(\xi,\eta)  \bra{  \rho_{N_1}(\eta) \nabla_\eta  \wh{f_{\theta}}(t,\eta), \al^\mu \wh{f_{\theta,N_2}}(t,\xi + \eta)} \,d\eta,\label{esti-timederi-1-2}\\
		&\frac{\rho_N(\xi)}{t} \int_{\R^2} e^{it\vp_{\theta}(\xi,\eta)} m_{\textbf{N}}(\xi,\eta)  \bra{  \rho_{N_1}(\eta)   \wh{f_{\theta}}(t,\eta), \al^\mu \nabla_\eta \wh{f_{\theta,N_2}}(t,\xi + \eta)} \,d\eta,\label{esti-timederi-1-3}
\end{align}
\end{subequations}
 where
\[
m_{\textbf{N}}(\xi,\eta) := \frac{\vp_{\theta}(\xi,\eta) \nabla_\eta \vp_{\theta}(\xi,\eta)}{|\nabla_\eta \vp_{\theta}(\xi,\eta)|^2}\rho_{N}(\xi)\rho_{N_1}(\eta)\rho_{N_2}(\xi+\eta)
\]
with $\textbf{N} =(N,N_1,N_2)$. Note that $f_\theo$ and $f_\thet$ are not symmetric from the support condition $N_2 \les N_1$. Let us consider $N \ll N_1 \sim N_2$. Then, we have the pointwise and Coifman-meyer bound  
\begin{align}\label{esti-timederi-multi}
	|\nabla_\eta m_{\textbf{N}}(\xi,\eta)|\les   N_1^{-1} \bra{N_1}^5  \;\;\mbox{ and }\;\; \|\nabla_\eta m_{\textbf{N}}\|_{\rm CM}   \les  N_1^{-1}\bra{N_1}^{11}.
\end{align}
Thus, for $N_1 \sim N_2 \le \bra{t}^{-\frac14}$, we use Lemma \ref{lem:coifmey}, Sobolev embedding $H^1 \hookrightarrow L^{p}$ $(2 \le p < \infty)$, and H\"older inequality and, for $N_1 \sim N_2 \ge \bra{t}^{-\frac14}$, we utilize Lemma \ref{prop:timedecay} and  Lemma \ref{lem:coifmey} with \eqref{esti-timederi-multi}. This yields that
\begin{align*}
	&\sum_{ N_1 \sim N_2}\| \eqref{esti-timederi-1-1}\|_{L_\xi^2} \\
	&\hspace{1cm}\les   	\sum_{N_1 \sim N_2 \le \bra{t}^{-\frac14}}t^{-1} \|\nabla_\eta m_{\textbf{N}}\|_{L_{\xi,\eta}^\infty}\|\rho_{N_1}\|_{L_\eta
	^\frac{4}{4-\gamma}}\|\wh{f_{\theta}}\|_{L_\eta^\frac4\gamma}\|\rho_{N_2}\|_{L_\xi^{\frac4{2-\gamma}}}\|\wh{f_{\theta}}\|_{L_\xi^{\frac4\gamma}} \\
&\hspace{2cm}+\sum_{ N_1 \sim N_2 \ge \bra{t}^{-\frac14}}t^{-1}N_1^{-1}\bra{N_1}^{11} \|\psi_{\theta,N_1}\|_{L^{\infty}}\|\psi_{\theta,N_2}\|_{L^2} \\
	&\hspace{1cm} \les 	\sum_{ N_1  \le \bra{t}^{-\frac14}} t^{-1}N_1^{2-\gamma} \|\bra{x}f_\theta\|_{H^5}\|\bra{x}f_\theta\|_{H^5}  \\
	&\hspace{2cm}+\sum_{  N_1  \ge \bra{t}^{-\frac14} }t^{-\frac74}N_1^{-1-\gamma}\bra{N_1}^{11-s}\|\psi_{\theta}\|_{W^{7,\infty}}\|\psi_{\theta,N_2}\|_{H^s} \\
	&\hspace{1cm} \les t^{-\frac32+\frac{\gamma}4}\ve_1^2. 
\end{align*}

For the case $N_2 \ll N_1$,  we may consider the multipliers regardless the sign relations. Indeed, we have the following bounds of multipliers
\begin{align*}
	|\nabla_\eta m_\textbf{N}(\xi,\eta)| \les N_2^{-1} \bra{N_1}    \;\;\mbox{ and } \;\; 	\|\nabla_\eta m_\textbf{N}\|_{\rm CM} \les N_2^{-1} \bra{N_1}.
\end{align*}
 Then, by H\"older inequality, we obtain
\begin{align*}
	\sum_{\substack{N_2 \ll N \sim N_1\\ N_2 \le \bra{t}^{-\frac14}}}	\| \eqref{esti-timederi-1-1}\|_{L_\xi^2}  &\les 	\sum_{\substack{N_2 \ll N \sim N_1\\ N_2 \le \bra{t}^{-\frac14}}}t^{-1}  \|\nabla_\eta m_\textbf{N}\|_{L_{\xi,\eta}^\infty} \|\rho_{N_2}\|_{L_\eta^{\frac2{2-\gamma}}}  \| \wh{f_{\theta,N_1}}\|_{L_\xi^2}\|\wh{f_{\theta}}\|_{L_\eta^{\frac2\gamma}} \\ 
	&\les 	\sum_{\substack{N_2 \ll N \sim N_1\\ N_2 \le \bra{t}^{-\frac14}}}  t^{-1}N_2^{-1}\bra{N_1}^{-4} N_2^{2-\gamma} \|\psi_{\theta}\|_{H^s} \|\bra{x}f_\theta\|_{H^5} \\
	&\les \bra{t}^{-\frac32 + \frac{\gamma}4}\ve_1^2,
\end{align*}
and by Lemma \ref{lem:coifmey}, we have
\begin{align*}
	\sum_{\substack{N_2 \ll N \sim N_1\\ N_2 \ge \bra{t}^{-\frac14}}}	\| \eqref{esti-timederi-1-1}\|_{L_\xi^2}  &\les 	\sum_{\substack{N_2 \ll N \sim N_1\\ N_2 \ge \bra{t}^{-\frac14}}}t^{-1} N_2^{-1} \bra{N_1} \|\psi_{\theo,N_1}\|_{L^\infty}   \|f_{\thet,N_2}\|_{L^2} \\ 
	&\les 	\sum_{\substack{N_2 \ll N \sim N_1\\ N_2 \ge \bra{t}^{-\frac14}}}  \bra{t}^{-\frac74}N_2^{-1}\bra{N_1}^{-6} \|\psi_\theo\|_{W^{7,\infty}} \|\psi_\thet\|_{H^s} \\
	&\les \bra{t}^{-\frac32 }\ve_1^2.
\end{align*}

Let us move on to the estimates for \eqref{esti-timederi-1-2}.  In a similar way to above, we have
\begin{align*}
	&\sum_{ N_1 \sim N_2}\| \eqref{esti-timederi-1-2}\|_{L_\xi^2} \\
	&\hspace{1cm}\les   	\sum_{N_1 \sim N_2 \le \bra{t}^{-\frac14}}t^{-1} \| m_{\textbf{N}}\|_{L_{\xi,\eta}^\infty}\|\rho_{N_1}\|_{L_\eta
		^2}\|x f_{\theta}\|_{L^2}\|\rho_{N_2}\|_{L_\xi^{\frac2{1-\gamma}}}\|\wh{f_{\theta,N_2}}\|_{L_\xi^{\frac2\gamma}} \\
	&\hspace{2cm}+\sum_{ N_1 \sim N_2 \ge \bra{t}^{-\frac14}}t^{-1}\bra{N_1}^{8} \|P_{N_1}(x f_{\theta})\|_{L^2}\|\psi_{\theta,N_2}\|_{L^\infty} \\
	&\hspace{1cm} \les 	\sum_{ N_1  \le \bra{t}^{-\frac14}} t^{-1}N_1^{2-\gamma}\ve_1^2  +\sum_{  N_1  \ge \bra{t}^{-\frac14} }\bra{t}^{-\frac74}\bra{N_1}^{-2}\ve_1^2 \\
	&\hspace{1cm} \les \bra{t}^{-\frac32 + \frac{\gamma}4}\ve_1^2. 
\end{align*}
For the case $N_2 \ll N_1$, we also get
\begin{align*}
	\sum_{N_2 \ll N \sim N_1}    \| \eqref{esti-timederi-1-2}\|_{L_\xi^2} &\les 	\sum_{\substack{N_2 \ll N \sim N_1\\ N_2 \le \bra{t}^{-\frac14}}}  t^{-1} \|m_{\textbf{N}}\|_{L_{\xi,\eta}^\infty}  \|P_{N_1}(xf_{\theo})\|_{L^2}\|\rho_{N_2}\|_{L_\eta^{\frac2{2-\gamma}}}\|f_{\thet,N_2}\|_{L_\eta^{\frac2\gamma}} \\
 	&\hspace{2cm}+ 	\sum_{\substack{N_2 \ll N \sim N_1\\ N_2 \ge \bra{t}^{-\frac14}}}t^{-1}  \bra{N_1} \|P_{N_1}(xf_{\theo})\|_{L^2}   \|\psi_{\thet,N_2}\|_{L^\infty} \\
	&\les  	\sum_{\substack{N_2 \ll N \sim N_1\\ N_2 \le \bra{t}^{-\frac14}}} t^{-1}N_2^{2-\gamma}\bra{N_1}^2\ve_1^2 + 	\sum_{\substack{N_2 \ll N \sim N_1\\ N_2 \ge \bra{t}^{-\frac14}}}\bra{t}^{-\frac74}  \bra{N_1}^{-4}\ve_1^2\\
	&\les   \bra{t}^{-\frac32 + \frac{\gamma}4} \ve_1^2.
\end{align*}
The estimates for \eqref{esti-timederi-1-3} can be obtained by following the same process used for the estimates for \eqref{esti-timederi-1-2}.

To get  the estimates for \eqref{esti-timederi-2}, \eqref{eq:esti-timederi-f} implies that
\begin{align*}
	\|\eqref{esti-timederi-2} \|_{L^2} \les N\bra{N}^{-5} \bra{t}^{-\frac32} \ve_1^2.
\end{align*}
This completes the proof of \eqref{eq:esti-timederi-bi}.
\end{proof}

\section{Proof of Theorem \ref{thm:mainthm}}\label{sec:3}

For any $\psi_{0}$ satisfying \eqref{eq:initial-condition}, it is quite standard to show the existence of local solution $\psi(t)$ in $\Sigma_{T}$ for some $T$ (for instance see \cite{choz2006-siam, lee}).
Then the solution to \eqref{eq:maineq} can be extended globally by a continuity
argument. To achieve this we have only to  prove the following: %%%%%%%%%%%%%%%%%%%%%%%%%%%%%%%%%%%%%%%%%%%%%%%%%%%%%%%%%%%%%%%%%%%%%%%%%%%%%%%%%%%%%%%%%%%%%%%%%%%%%%%%%%%%%%%%%%%%%%%%%%%%%%%%%%%%%%%%%%%%
Given any $T > 0$, let $\psi$ be a solution with initial data satisfying \eqref{eq:initial-condition} on $[0,T]$. For a small $\varepsilon_1 > 0$,  we assume that
\[
\|\psi\|_{\Sigma_{T}}\le K\ve_{1}.
\]
Then there exists $C$ depending only on $K$ such that
\begin{align}
	\|\psi\|_{\Sigma_{T}} \le \ve_{0} + C \ve_{1}^{3}.\label{claim}
\end{align}
The proof of \eqref{claim} will be done if we prove  Propositions \ref{prop:energy-high} and \ref{prop:energy-weight}  below. Recall that  we denote $\Pi_{\theta}(D)\psi$ by $\psi_{\theta}$ for $\theta\in\{+,-\}$, so 
\begin{align*}
	\psi=\psi_+ + \psi_-.
\end{align*}

\begin{prop}\label{prop:energy-high} 
	Assume that $\psi\in C([0,T],H^s)$ satisfies
	the condition \eqref{eq:apriori}.
	Then we obtain the following estimates: For $\theta_0\in\{+,-\}$,
	\begin{align*}
	 \sup_{t\in[0,T]}\|\psi_{\thez}(t)\|_{H^{s}}\le\ve_{0}+C\ve_{1}^{3}.
	\end{align*}
\end{prop} 
The proof of this proposition is in the end of this section.

\begin{prop}[Weighted estimates]\label{prop:energy-weight}
	Let $\psi\in C([0,T];H^{s})$ satisfy a priori assumption \eqref{eq:apriori}. Then we have
	\begin{align}
		\sup_{t_{1}\le t_{2}\in[0,T]}\bra{t_1}^{\de}\normo{\bra{x}\Big(f_{\theta}(t_{2})-f_{\theta}(t_{1})\Big)}_{H^5}\le\ve_{1}\label{bound-linfty}
	\end{align}
	for sufficiently small  $\de>0$ and $\theta \in \{ +,-\}$. 
\end{prop}
\noindent The proof of this proposition
constitutes the main part of our analysis.  This proposition will be
proved in section \ref{sec:mainproof}.
The bound \eqref{bound-linfty} implies that 
the global solution $\psi_{\theta}=\Pi_{\theta}(D)\psi$ converges to a scattering profile $\phi_{\theta}$ defined by
\[
\phi_{\theta}(\xi):= \mathcal F^{-1}\left(\lim_{t\to\infty}f_{\theta}(t,\xi)\right).
\]
Then, \eqref{bound-linfty} leads us to
\begin{align}
	\normo{\psi_{\theta}(t)- e^{-\theta it\bra{D}}\phi_{\theta}}_{H^5(\bra{x}^2 dx)}\les\bra{t}^{-\de}\ve_1.\label{scatt-decou}
\end{align}
Setting $\phi:=\phi_{+}+\phi_{-}$, \eqref{scatt-decou} implies the
scattering results \eqref{aim:scattering}. This completes
the proof of Theorem \ref{thm:mainthm}.

\subsection{Proof of Energy estimates}\label{sec:sobolev}

By \eqref{eq:timedecay} and H\"older inequality, we estimate
\begin{align*}
	&\|\bra{D}^{-1}\left[\bra{\psi_\thet(s),\al^\mu\psi_\theth(s)}\right] \al^\nu\psi_\theo(s)\|_{H^s} \\
	&\les \|\bra{\psi_\thet(s),\al^\mu\psi_\theth(s)}\|_{H^{s
	-1}} \|\psi_\theo(s)\|_{L^\infty} + \|\bra{\psi_\thet(s),\al^\mu\psi_\theth(s)}\|_{L^\infty}\|\psi_\theo(s)\|_{H^s}\\
	&\les \bra{s}^{-\frac32}\ve_1^3,
\end{align*}
which implies that
\begin{align*}
	\|\psi_\thez(t)\|_{H^s} &\le \|\psi_0\|_{H^s}  + \wt{C}\sum_{\thej \in \{\pm\},j=1,2,3}\int_0^t 	\|\bra{D}^{-1} \left[\bra{\psi_\thet(s), \al^\mu\psi_\theth(s)} \right]\al^\nu\psi_\theo(s)\|_{H^s} ds \\
	& \le \ve_0 + C\ve_1^3.
\end{align*}
This finishes the proof of Proposition \ref{prop:energy-high}.

\section{Proof of weighted estimates}\label{sec:mainproof}
This section is devoted to prove Proposition \ref{prop:energy-weight}.  From the fact that
\begin{align*}
	\bra{x}\Big(f_{\thez}(t_{2})-f_{\thez}(t_{1})\Big) = \bra{x}\int_{t_1}^{t_2} \partial_s f_\thez(s) ds,
\end{align*}
and
\begin{align*}
	\partial_s f_\thez(s) = C \Pi_\thez(D) \sum_{\thej \in \{\pm\},j=1,2,3} \bra{D}^{-1}\left[\bra{\psi_\thet(s),\al^\mu \psi_\theth(s)} \right]\al^\nu \psi_\theo(s), 
\end{align*}
 it suffices to show that
\begin{align}\label{eq:decay-nonlinearity}
	\normo{ \int_{t_1}^{t_2}  x  \left[\Pi_{\thez}(D) \bra{D}^{-1} \left[\bra{\psi_\thet(s), \al^\mu\psi_\theth(s)}\right] \al^\nu\psi_\theo(s)  \right] ds}_{H^5} \les \bra{t_1}^{-\de}\ve_1^3.
\end{align}
Note that $\|x \phi\|_{H^5}\sim \|\bra{\xi}^5\mathcal{F}\left(x \phi \right)\|_{L_{\xi}^{2}}\sim \left\|\bra{\xi}^5\nabla_{\xi}\phi \,\right\|_{L_{\xi}^{2}}.
$ Then,  \eqref{eq:decay-nonlinearity} is equivalent to
\begin{align*}
\left\| \int_{t_1}^{t_2}  \bra{\xi}^5  \sum_{j=1}^3\mathcal I_\Theta^j(s,\xi) ds \right\|_{L_\xi^2}\les  \bra{t_1}^{-\de} \ve_1^3,
\end{align*}
for sufficiently small $\de>0$, where the integrands 
\begin{align*}
		\mathcal I_\Theta^{1}(s,\xi) & = \int_{\mathbb{R}^{2}}\nabla_\xi \left[\Pi_{\thez}(\xi)\right] \bra{\eta}^{-1}e^{is{\phi_{01}}(\xi,\eta)}\mathcal{F}(\bra{\psi_\thet(s), \al^\mu\psi_\theth(s)})(\eta)  \al^\nu\widehat{f_\theo}(s,\xi-\eta) \,d\eta ,\\
	\mathcal I_\Theta^{2}(s,\xi) & = \int_{\mathbb{R}^{2}}\Pi_{\thez}(\xi) \bra{\eta}^{-1}e^{is{\phi_{01}}(\xi,\eta)}\mathcal{F}(\bra{\psi_\thet(s), \al^\mu\psi_\theth(s)})(\eta)  \al^\nu\nabla_{\xi}\widehat{f_\theo}(s,\xi-\eta) \,d\eta ,\\
	\mathcal I_\Theta^{3}(s,\xi) & = s\int_{\mathbb{R}^{2}}\mathbf{m}(\xi,\eta)e^{is\phi_{01}(\xi,\eta)}\mathcal{F}(\bra{\psi_\thet(s), \al^\mu \psi_\theth(s)})(\eta) \al^\nu\widehat{f_\theo}(s,\xi-\eta) \,d\eta ,
\end{align*}
with the phase function $\phi_{01}$ and multiplier $\textbf{m}$,
\begin{align*}
	\begin{aligned} %\label{eq:pm}
		\phi_{01}(\xi,\eta) & = \thez\langle\xi\rangle- \theo\langle\xi-\eta\rangle,\\
		\mathbf{m}(\xi,\eta) & = \Pi_{\thez}(\xi)\bra{\eta}^{-1}\nabla_{\xi}\phi_{01}(\xi,\eta)=\Pi_{\thez}(\xi) \left(\thez\frac{\xi}{\langle\xi\rangle}-\theo\frac{\xi-\eta}{\langle\xi-\eta\rangle} \right).
	\end{aligned}
\end{align*}
Here, we denoted the 4-tuple of $(\thez,\theo,\thet,\theth)$ by $\Theta$. We may treat  $\mathcal I_\Theta^1$  similarly to the energy estimates in section \ref{sec:sobolev}. For the estimate of $\mathcal I_\Theta^2$, by \eqref{eq:timedecay}, we readily have
\begin{align*}
\|\bra{\xi}^5\mathcal I_\Theta^2(s,\xi)\|_{L_\xi^2} \les  \|\bra{\psi_\thet(s),\al^\mu\psi_\theth(s)}\|_{W^{5,\infty}} \|xf_\theo (s)\|_{H^5}  \les \bra{s}^{-\frac32} \ve_1^3,
\end{align*}
which implies that
\begin{align*}
\int_{t_1}^{t_2}	\|\bra{\xi}^5\mathcal I_\Theta^2(s)\|_{L_\xi^2} \,ds \les  \bra{t_1}^{-\frac12} \ve_1^3.
\end{align*}

Let us move on to $\mathcal I_\Theta^3(s,\xi)$.   In view of Lemma \ref{lem:esti-l2-decay}, we consider the estimates for $\mathcal I_\Theta^3$ by dividing the resonant case and non-resonant case.

\subsection{Non-resonant case: $\thet = \theth$}\label{sec:res}  For the purpose of obtaining the estimates for this case, we decompose the frequencies $|\xi|$, $|\xi-\eta|$ and $|\eta|$ into dyadic numbers $N_0$, $N_1$ and $N_2$, respectively. We denote $(N_0, N_1, N_2)$ as $\textbf{N}$ in sequel. Let us now consider
\begin{align*}
	\normo{ \int_{t_1}^{t_2}\bra{\xi}^5\mathcal I_\Theta^3(s,\xi) \,ds}_{L_\xi^2} \les \sum_{\textbf{N}} \normo{ \int_{t_1}^{t_2}\mathcal I_{\Theta, \textbf{N}}^3(s,\xi) \,ds}_{L_\xi^2}, 
\end{align*}
where 
\[
\mathcal I_{\Theta, \textbf{N}}^3(s,\xi) = s\int_{\mathbb{R}^{2}}\mathbf{m}_{\textbf{N}}(\xi,\eta)e^{is\phi_{01}(\xi,\eta)}\mathcal{F}\Big( P_{N_2}\bra{\psi_\thet(s), \al^\mu \psi_\theth(s)}\Big)(\eta) \al^\nu\widehat{f_{\theo,N_1}}(s,\xi-\eta) \,d\eta 
\]
with the multiplier
\[
\mathbf{m}_{\textbf{N}}(\xi,\eta) =\Pi_{\thez}(\xi) \bra{\xi}^5\bra{\eta}^{-1} \left(\thez\frac{\xi}{\langle\xi\rangle}-\theo\frac{\xi-\eta}{\langle\xi-\eta\rangle} \right) \rho_{N_0}(\xi)\rho_{N_1}(\xi-\eta)\rho_{N_2}(\eta).
\]
In the rest of this section, we focus on to prove 
\begin{align}\label{eq:decay-fristmoment}
	\normo{	\int_{t_1}^{t_2}   \mathcal I_{\Theta,\textbf{N}}^3 (s,\xi) \,ds }_{L_\xi^2} \les \bra{t_1}^{-\de}\ve_1^3.
\end{align}
To this end, we classify the support condition into following three cases: 
 \begin{itemize}
 \item \textbf{Case (i)}: $N_0 \ll N_1 \sim N_2$.
 \item \textbf{Case (ii)}: $N_1 \ll N_0 \sim N_2$.
  \item \textbf{Case (iii)}: $N_2 \ll N_0 \sim N_1$. 
\end{itemize}
  \textbf{Estimates for Case (i)}. By H\"older inequality, we have
\begin{align*}
	&\int_{t_1}^{t_2} \sum_{\substack{\textbf{Case (i)} \\ N_0 \le \bra{s}^{-3}}}s  \bra{N_0}^{5}\bra{N_2}^{-1} \normo{\rho_{N_0}}_{L_\xi^2} \|P_{N_2}\bra{\psi_\thet(s),  \al^\mu\psi_\theth(s)}\|_{L^2}  \bra{N_1}^{-s}\|\psi_\theo\|_{H^s} ds\\
	&\qquad\les  \int_{t_1}^{t_2} \sum_{\substack{\textbf{Case (i)} \\ N_0 \le \bra{s}^{-3}}} s N_0 \bra{N_1}^{-s-1}\prod_{j=1}^3\|\psi_{\thej}\|_{H^s} \,ds \les \bra{t_1}^{-1}\ve_1^3.
\end{align*}
Using \eqref{eq:esti-l2-decay} with $0<\gamma=\de \ll 1$ and Lemma \ref{lem:coifmey} with the estimate
\begin{align}\label{eq:multi-case-1}
	\|\textbf{m}_{\textbf{N}}\|_{\rm CM} \les \bra{N_0}^5 \frac{N_1}{\bra{N_1}^{2}},	
\end{align}
we see that
\begin{align*}
	&\int_{t_1}^{t_2} \sum_{\substack{\textbf{Case (i)}\\ N_0 \ge \bra{s}^{-3}}}s  \bra{N_0}^{5} \frac{N_1}{\bra{N_1}^{2}}\|P_{N_2}\bra{\psi_\thet(s), \al^\mu \psi_\theth(s)}\|_{L^2}  \bra{N_1}^{-7}\|\psi_\theo\|_{W^{7,\infty}} \,ds\\
	&\qquad \les  \int_{t_1}^{t_2}  \sum_{\substack{\textbf{Case (i)}\\ N_0 \ge \bra{s}^{-3}}} \bra{s}^{-\frac32 +\de}  \bra{N_2}^{-4}\ve_1^3 \,ds \les \bra{t_1}^{-\frac12 +2\de}\ve_1^3.
\end{align*}
Note that \eqref{eq:multi-case-1} is independent of the sign relation. Then this completes the proof of  \eqref{eq:decay-fristmoment} for \textbf{Case (i)}.

\textbf{Estimates for Case (ii)}: $N_1 \ll N_0 \sim N_2$. In this case, we estimate by dividing further into $N_0 \sim N_2 \ge \bra{s}^{\frac3s}$ and $N_0 \sim N_2 \le \bra{s}^{\frac3s}$ to handle $\bra{N_0}^k$. Indeed, by Lemma \ref{lem:coifmey}, Young's inequality  and Sobolev embedding $H^1 \hookrightarrow L^{p}(2\le p <\infty)$, we have
\begin{align*}
	& \int_{t_1}^{t_2} \sum_{\textbf{Case (ii)}} \normo{\mathcal I_{\Theta, \textbf{N}}^3(s,\xi)}_{L_\xi^2} ds \\
	&\les
	% \sum_{\substack{\textbf{Case (iii)}\\ N_1 \le \bra{s}^{-1} }} s \frac{\bra{\xi}^{5-\ell}N_0}{\bra{N_0}} \normo{P_{N_2} \bra{\psi_\thet(s),\al^\mu \psi_{\theth}(s)}}_{L^\infty} \|\psi_{\theo,N_1}\|_{L^2}\\
	\int_{t_1}^{t_2} \sum_{\substack{\textbf{Case (ii)}\\ N_0 \ge \bra{s}^{\frac3s}}} s \bra{N_0}^{3}N_0 \normo{P_{N_2} \bra{\psi_\thet(s),\al^\mu \psi_{\theth}(s)}}_{L^\infty} \|\psi_{\theo,N_1}\|_{L^2} ds\\
	&\les 	\int_{t_1}^{t_2}	\sum_{\substack{\textbf{Case (ii)}\\ N_0 \ge \bra{s}^{\frac3s}}} s \bra{N_0}^{4-s} \|\psi_\thet\|_{H^s} \|\psi_\theth\|_{H^s}  N_1^{\de}  \|\bra{x} f_{\theo}\|_{L^2} ds \les  \bra{t_1}^{-1}\ve_1^3.
\end{align*}
 On the other hand, Lemma \ref{lem:coifmey}, Bernstein's inequality, and \eqref{eq:esti-l2-decay}  lead us that
\begin{align*}
	&\int_{t_1}^{t_2}\sum_{\substack{\textbf{Case (ii)}\\ N_0 \le \bra{s}^{\frac3s}}}\normo{\mathcal I_{\Theta, \textbf{N}}^3(s,\xi)}_{L_\xi^2}  ds \\
	&\les
	% \sum_{\substack{\textbf{Case (iii)}\\ N_1 \le \bra{s}^{-1} }} s \frac{\bra{\xi}^{5-\ell}N_0}{\bra{N_0}} \normo{P_{N_2} \bra{\psi_\thet(s),\al^\mu \psi_{\theth}(s)}}_{L^\infty} \|\psi_{\theo,N_1}\|_{L^2}\\
	\int_{t_1}^{t_2} \sum_{\substack{\textbf{Case (ii)}\\ N_0 \le \bra{s}^{\frac3s}}} s \bra{N_0}^{3}N_0 \normo{P_{N_2} \bra{\psi_\thet(s),\al^\mu \psi_{\theth}(s)}}_{L^2} \|\psi_{\theo,N_1}\|_{L^\infty} ds\\
	&\les 	\int_{t_1}^{t_2}	\sum_{\substack{\textbf{Case (ii)}\\ N_0 \le \bra{s}^{\frac3s}}} \bra{s}^{-\frac32 + \de} \bra{N_0}^{3} N_1^{\de} \ve_1^3 ds \les  \bra{t_1}^{-\frac12 +2\de}\ve_1^3.
\end{align*}
These finish the proof of \eqref{eq:decay-fristmoment} for \textbf{Case (ii)}.

\textbf{Estimates for Case (iii):} $N_2 \ll N_0 \sim N_1$.   For this case, we also further decompose into $N_2 \le \bra{s}^{-1}$ and $N_2 \ge \bra{s}^{-1}$.  For the first case $N_2 \le \bra{s}^{-1}$, by H\"older inequality, we have
\begin{align*}
	\int_{t_1}^{t_2} \sum_{\substack{\textbf{Case (iii)}\\ N_2 \le \bra{s}^{-1} }}\normo{\bra{\xi}^5\mathcal I_{\Theta, \textbf{N}}^3(s,\xi)}_{L_\xi^2} ds &\les 	\int_{t_1}^{t_2} \sum_{\substack{\textbf{Case (iii)}\\ N_2 \le \bra{s}^{-1} }} s\bra{N_0}^5N_2 \|\rho_{N_2}\|_{L_\eta^1} \bra{N_1}^{-s}\prod_{j=1}^3 \|\psi_\thej\|_{H^s} ds   \\
	&\les  \bra{t_1}^{-2+\de}\ve_1^3.
\end{align*}
Then, we have \eqref{eq:decay-fristmoment} for $N_2 \le \bra{s}^{-1}$. We now shall need to estimate for $N_2 \ge \bra{s}^{-1}$. In order to obtain the desired bound for $N_2 \ge \bra{s}^{-1}$, we describe a null structure with respect to the sign relation between $\thez$ and $\theo$.
\begin{rem}\label{rem:theta01}
We will exploit Lemma \ref{lem:esti-l2-decay} to estimate $\mathcal I_{\Theta,\textbf{N}}^3$ for the case $N_2 \ge \bra{s}^{-1}$. Lemma \ref{lem:esti-l2-decay} possesses the extra time decay and it causes the singularity of $N_2$ simultaneously. For this reason, we shall need some structure, which can kill the singularity. If $\thez =\theo$, $\nabla_\xi \phi_{01}$ in our multiplier $\textbf{m}_\textbf{N}$ implies the $\eta$-smallness that plays a role of eliminating the singularity. Otherwise, if $\thez \neq \theo$, there is no such null structure in $\textbf{m}_\textbf{N}$. Fortunately, since the phase $\phi_{01}$ has time nonresonance, we can utilize the normal form approach to obtain the extra time decay. Therefore, it is natural to divide the case $\thez = \theo$ and $\thez \neq \theo$.
\end{rem}
As we mentioned above remark, if $\thez = \theo$, $\nabla_\xi \phi_{01}$ leads us that
\begin{align*}
	\|\textbf{m}_{\textbf{N}} \|_{\rm CM} \les \frac{N_2}{\bra{N_0}}.
\end{align*}
Then, by Lemma \ref{lem:coifmey} and \eqref{eq:esti-l2-decay}, we get
\begin{align*}
	&\int_{t_1}^{t_2}\sum_{\substack{\textbf{Case (iii)}\\ N_2 \ge \bra{s}^{-1} }}	\normo{\bra{\xi}^5\mathcal I_{\Theta, \textbf{N}}^3(s,\xi)}_{L_\xi^2} ds\\ 
	&\les \int_{t_1}^{t_2} \sum_{\substack{\textbf{Case (iii)}\\ N_2 \ge \bra{s}^{-1} }} s\bra{N_0}^4N_2 \| P_{N_2}\bra{\psi_\thet(s),  \al^\mu\psi_\theth(s)}\|_{L^2} \bra{N_1}^{-k}\|\psi_{\theo}(s)\|_{W^{k,\infty}} ds\\
	& \les \int_{t_1}^{t_2}  \sum_{\bra{s}^{-1} \le N_2 \ll N_0} \bra{s}^{-\frac32+\de} \bra{N_0}^{-3} \ve_1^3  ds \les \bra{t_1}^{-\frac12+2\de}\ve_1^3.
\end{align*}

If $\thez \neq \theo$, we  have $|\nabla_\xi \phi_{01}| \sim 1$, i.e., there is no null structure to handle the singularity $N_2$. As we mentioned in Remark \ref{rem:theta01}, we do not use any null structure, and instead, we use the normal form approach by exploiting the time nonresonance
\begin{align*}
|\phi_{01}| = |\bra{\xi} + \bra{\xi-\eta}| \sim \bra{N_0}.
\end{align*}
Using a relation
\[
e^{is\phi_{01}} = \frac{\partial_s e^{is\phi_{01}}}{i\phi_{01}}
\]
in $\mathcal I_{\Theta,\textbf{N}}^3$, we see that integration by parts in time yields the following terms:
\begin{subequations}
	\begin{align}
		&t_2\int_{\mathbb{R}^{2}} \wt{\textbf{m}_{\textbf{N}}} (\xi,\eta)e^{it_2\phi_{01}(\xi,\eta)}\mathcal{F}(P_{N_2}\bra{\psi_\thet(t_2), \al^\mu \psi_\theth(t_2)})(\eta) \al^\nu\widehat{f_{\theo,N_1}}(t_2,\xi-\eta) \,d\eta , \label{eq:normal-t2}\\
		&t_1\int_{\mathbb{R}^{2}}\wt{\textbf{m}_{\textbf{N}}} (\xi,\eta)e^{it_1\phi_{01}(\xi,\eta)}\mathcal{F}(P_{N_2}\bra{\psi_\thet(t_1), \al^\mu \psi_\theth(t_1)})(\eta) \al^\nu\widehat{f_{\theo,N_1}}(t_1,\xi-\eta) \,d\eta , \label{eq:normal-t1}\\
		&\int_{t_1}^{t_2}  \int_{\mathbb{R}^{2}}\wt{\textbf{m}_{\textbf{N}}} (\xi,\eta)e^{is\phi_{01}(\xi,\eta)}\left[\mathcal{F}(P_{N_2}\bra{\psi_\thet(s), \al^\mu \psi_\theth(s)})(\eta)\right] \al^\nu\widehat{  f_{\theo,N_1}}(s,\xi-\eta) \,d\eta  ds, \label{eq:normal-s}\\
		&\int_{t_1}^{t_2} s \int_{\mathbb{R}^{2}}\wt{\textbf{m}_{\textbf{N}}} (\xi,\eta)e^{is\phi_{01}(\xi,\eta)}\partial_s\left[\mathcal{F}(P_{N_2}\bra{\psi_\thet(s), \al^\mu \psi_\theth(s)})(\eta)\right] \al^\nu\widehat{  f_{\theo,N_1}}(s,\xi-\eta) \,d\eta  ds, \label{eq:normal-f23}\\
		&\int_{t_1}^{t_2} s \int_{\mathbb{R}^{2}}\wt{\textbf{m}_{\textbf{N}}} (\xi,\eta)e^{is\phi_{01}(\xi,\eta)}\mathcal{F}(P_{N_2}\bra{\psi_\thet(s), \al^\mu \psi_\theth(s)})(\eta) \al^\nu\widehat{ \partial_s f_{\theo,N_1}}(s,\xi-\eta) \,d\eta  ds, \label{eq:normal-f1}
	\end{align}
\end{subequations}
where
\[
\wt{\textbf{m}_{\textbf{N}}} (\xi,\eta) := \frac{\mathbf{m}_{\textbf{N}}(\xi,\eta)}{\phi_{01}(\xi,\eta)}.
\]
Then we get
\begin{align}\label{eq:multiplier-tilde}
	\|\wt{\textbf{m}_\textbf{N}}\|_{\rm CM} \les \bra{N_0}^3.
\end{align}
To bound  \eqref{eq:normal-t2}, we use Lemma \ref{lem:coifmey} with \eqref{eq:multiplier-tilde} to obtain
\begin{align*}
	\sum_{\substack{\textbf{Case (iii)}\\ N_2 \ge \bra{t_2}^{-1} }}\normo{\eqref{eq:normal-t2}}_{L_\xi^2} &\les \sum_{\substack{\textbf{Case (iii)}\\ N_2 \ge \bra{t_2}^{-1} }} t_2 \bra{N_0}^3 \|P_{N_2}\bra{\psi_\thet(t_2), \al^\mu \psi_\theth(t_2}\|_{L^\infty} \bra{N_1}^{-s} \|\psi_{\theo}\|_{H^s}\\
	&\les \bra{t_2}^{-\frac12 + \de} \ve_1^3 \les \bra{t_1}^{-\de} \ve_1^3. 
\end{align*}
For the estimates of \eqref{eq:normal-t1} and \eqref{eq:normal-s}, we may have the desired results in a similar way to the estimates for $\eqref{eq:normal-t2}$. To estimate \eqref{eq:normal-f23} and \eqref{eq:normal-f1}, we utilize Lemma \ref{lem:esti-timederi}. Indeed, using \eqref{eq:esti-timederi-bi} and Lemma \ref{lem:coifmey} with \eqref{eq:multiplier-tilde}, we see that
\begin{align*}
	\sum_{\substack{\textbf{Case (iii)}\\ N_2 \ge \bra{s}^{-1} }}\normo{\eqref{eq:normal-f23}}_{L_\xi^2} &\les  \int_{t_1}^{t_2}s \sum_{\substack{\textbf{Case (iii)}\\ N_2 \ge \bra{s}^{-1} }} \bra{N_0}^3 \|\partial_s P_{N_2}\bra{\psi_\thet(s), \al^\mu \psi_\thet(s)}\|_{L^2} \|\psi_{\theo,N_1}(s)\|_{L^\infty}   \,ds \\
	&\les  \int_{t_1}^{t_2} s \sum_{\substack{\textbf{Case (iii)}\\ N_2 \ge \bra{s}^{-1} }} \bra{N_0}^3   \bra{s}^{-\frac32 + \frac\de4}\bra{N_1}^{-7}\bra{s}^{-\frac34} \ve_1^3  \,ds \les \bra{t_1}^{-\frac14+2\de} \ve_1^3. 
\end{align*}
The estimate for \eqref{eq:normal-f1} can be treated in a similar way to above with \eqref{eq:esti-timederi-f} and we finish the proof of \eqref{eq:decay-fristmoment} for \textbf{Case (iii)}.

\subsection{Resonant case: $\thet \neq \theth$}\label{sec:non-res}  For the purpose of obtaining estimates for $\mathcal I_\Theta^3(s,\xi)$, we further take Fourier transform and the suitable change of variables. Then we get
\begin{align*}
	\mathcal I_{\Theta}^3(s,\xi) &= s\iint_{\mathbb{R}^{2}\times \R^2}\mathbf{m}(\xi,-\eta)e^{is\phi_\Theta(\xi,\eta,\sigma)} \bra{\wh{f_{\thet}}(s,\xi+\eta+\sigma), \al^\mu \wh{f_{\theth}}(s,\xi +\sigma)} \al^\nu\widehat{f_{\theo}}(s,\xi+\eta) \,d\sigma d\eta, 	 	
\end{align*}
where
\begin{align*}
	\phi_\Theta (\xi,\eta,\sigma) = \thez \bra{\xi} -\theo\bra{\xi+\eta} +\thet\Big(\bra{\xi+\eta+\sigma} +\bra{\xi +\sigma}\Big)    \;\;\mbox{ for } \; \Theta= (\thez,\theo,\thet,-\thet)
\end{align*}
and decompose $|\xi|$, $|\xi+\eta|$, $|\xi+\eta+\sigma|$ and $|\xi+\sigma|$ into  dyadic numbers $N_0,N_1,N_2$ and $N_3$, respectively. Moreover, we  decompose $|\eta|$, $|\sigma|$ into $L_1$, $L_2$ respectively. Let us denote $(N_0,N_1,N_2)$ and $(L_1,L_2)$ as $\textbf{N}$ and $\textbf{L}$. Note the difference between the frequency supports in section \ref{sec:res}. Then we see that
\begin{align*}
	\normo{ \int_{t_1}^{t_2}\mathcal I_\Theta^3(s,\xi) \,ds}_{L_\xi^2} = \normo{ \int_{t_1}^{t_2} \sum_{\textbf{N},\textbf{L}} \mathcal I_{\Theta, \textbf{N},\textbf{L}}^3(s,\xi) \,ds}_{L_\xi^2}, 
\end{align*}
where 
\begin{align*}
\mathcal I_{\Theta, \textbf{N},\textbf{L}}^3(s,\xi) &= s\iint_{\mathbb{R}^{2}\times \R^2}\mathbf{m}_{\textbf{N},\textbf{L}}(\xi,\eta)e^{is\phi_\Theta(\xi,\eta,\sigma)} \bra{\wh{f_{\thet,N_2}}(s,\xi+\eta+\sigma), \al^\mu \wh{f_{-\thet,N_3}}(s,\xi +\sigma)}\\
&\hspace{8.5cm}\times  \al^\nu\widehat{f_{\theo,N_1}}(s,\xi+\eta) \,d\sigma d\eta,
\end{align*}
with the multiplier
\[
\mathbf{m}_{\textbf{N},\textbf{L}}(\xi,\eta) =\Pi_{\thez}(\xi) \bra{\xi}^5 \bra{\eta}^{-1} \left(\thez\frac{\xi}{\langle\xi\rangle}-\theo\frac{\xi+\eta}{\langle\xi+\eta\rangle} \right) \rho_{\textbf{N},\textbf{L}}(\xi,\eta,\sigma).
\]
Here $\rho_{\textbf{N},\textbf{L}}(\xi,\eta,\sigma)$ denotes a bundle of cut-off functions with respect to $N_\mu$'s and $L_j$'s. Then we devote to show that 
\begin{align}\label{eq:decay-non-res}
	\normo{	\int_{t_1}^{t_2} \sum_{\textbf{N},\textbf{L}}   \mathcal I_{\Theta,\textbf{N},\textbf{L}}^3 (s,\xi) \,ds }_{L_\xi^2} \les \bra{t_1}^{-\de}\ve_1^3.
\end{align}

As we obtained the similar result to \eqref{eq:decay-non-res} in section \ref{sec:res}, we need to estimate  \eqref{eq:decay-non-res} by decomposing the support condition cases. Let us first consider two cases $L_1 \ll N_2 \sim N_3$ and $N_2 \nsim N_3$. The significant difficulty, compared to section \ref{sec:res} is that we cannot use the non-resonant estimates \eqref{eq:esti-l2-decay} in Lemma \ref{lem:esti-l2-decay} and \eqref{eq:esti-timederi-bi} in Lemma \ref{lem:esti-timederi} from the fact that these estimates were only proved when signs of the two spinors $\psi_\thet$ and $\psi_\theth$ are the same. However, we can verify the same results for the cases $L_1 \ll N_2 \sim N_3$ or $N_2 \nsim N_3$, since one obtains same lower bounds in \eqref{eq:space-res-1} and \eqref{eq:space-res-2},  which plays  a crucial role in the proof of \eqref{eq:esti-l2-decay} and \eqref{eq:esti-timederi-bi}.  Indeed, we have following lemma.
\begin{lemma}
	Let $\thet \neq \theth$ and $L_1 \ll N_2 \sim N_3$ or $N_2 \nsim N_3$.	Assume that $\psi$ satisfies a priori assumption \eqref{eq:apriori}. Then we have
		\begin{align}
				&\|P_{L_1} \bra{\psi_{\thet,N_2}(s),\al^\mu\psi_{\theth,N_3}(s)}\|_{L_x^2} \les L_1^{-1} s^{-\frac74}\ve_1^2,\label{eq:4-l2-decay}\\
			& \|P_{L_1} \partial_s \bra{\psi_{\thet,N_2}(s), \al^\mu \psi_{\theth,N_3}(s)}\|_{L^2} \les  \bra{s}^{-\frac32+\frac{\gamma}4} \ve_1^2 \label{eq:4-time-deri}
		\end{align}
		for sufficiently small $0 < \gamma \ll 1$.
\end{lemma}

\begin{rem}
	The estimates \eqref{eq:4-l2-decay} and \eqref{eq:4-time-deri} correspond to \eqref{eq:esti-l2-decay} and \eqref{eq:esti-timederi-bi}, respectively. In the proofs of \eqref{eq:esti-l2-decay} and \eqref{eq:esti-timederi-bi}, it is necessary to obtain the lower bounds of the resonance functions as \eqref{eq:space-res-1} and \eqref{eq:space-res-2}. However, since we have the same lower bounds for the cases $L_1 \ll N_2 \sim N_3$ and $N_2 \nsim N_3$,  one can obtain the proofs of \eqref{eq:4-l2-decay} and \eqref{eq:4-time-deri} in a similar way, by following the proofs of \eqref{eq:esti-l2-decay} and \eqref{eq:esti-timederi-bi}. Thus, we omit the details.
\end{rem}

Once we assume the validity of above lemma, this lemma enables us to obtain \eqref{eq:decay-non-res} for cases $L_1 \ll N_2 \sim N_3$ and $N_2 \nsim N_3$ in a similar way to section \ref{sec:res}, straightforwardly. Then, it remains only to deal with the support conditions 
\begin{align}\label{eq:condi-support}
	L_1 \sim N_2 \sim N_3.
\end{align}
To this end, by the support condition, we may consider the four cases:
\begin{itemize}
 \item \textbf{Case (i)}: $N_0 \ll N_1 \sim L_1$.
 \item \textbf{Case (ii)}: $N_1 \ll N_0 \sim L_1$.
  \item \textbf{Case (iii)}: $L_1 \ll N_0 \sim N_1$.
  \item \textbf{Case (iv)}: $N_0 \sim N_1 \sim L_1$.
\end{itemize}
\subsubsection{Estimates for \textbf{Case (i): $N_0 \ll N_1 \sim L_1$}.}\label{sec:subsub1} Note that the frequency relation between $\xi$, $\xi +\sigma$, and $\sigma$ implies that $N_0 \ll N_3 \sim L_2$. Then, due to the constraint \eqref{eq:condi-support}, we have 
$$
N_0 \ll N_1 \sim N_2 \sim N_3 \sim L_1 \sim L_2.
$$
By H\"older inequality, we readily get
\begin{align*}
%		&\normo{	\int_{t_1}^{t_2} \sum_{\substack{\textbf{Case (i)} \\ N_0 \le \bra{s}^{-2}}}   \mathcal I_{\Theta,\textbf{N},\textbf{L}}^3 (s,\xi) \,ds }_{L_\xi^2} \\
		& \sum_{\substack{\textbf{Case (i)} \\ N_0 \le \bra{s}^{-2}}}s  \bra{N_0}^5 \bra{N_2}^{-1} \normo{\rho_{N_0}}_{L_\xi^2} \|P_{N_2}\bra{\psi_\thet(s),  \al^\mu\psi_{\theth}(s)}\|_{L^2}  \bra{N_1}^{-s}\|\psi_\theo\|_{H^s}\\
	&\qquad\les  \sum_{\substack{\textbf{Case (i)} \\ N_0 \le \bra{s}^{-2}}} s N_0 \bra{N_1}^{-s}\|\psi_{\theth}\|_{L^\infty}\|\psi_{\thet}\|_{H^s} \|\psi_{\theo}\|_{H^s}\les \bra{s}^{-\frac74}\ve_1^3,
\end{align*}
which leads us that
\begin{align*}
	\int_{t_1}^{t_2} \sum_{\substack{\textbf{Case (i)}\\ N_0 \le \bra{s}^{-2}}}  \normo{ \mathcal I_{\Theta,\textbf{N},\textbf{L}}^3}_{L_\xi^2} \,ds \les \bra{t_1}^{-\frac34}\ve_1^3.
\end{align*}
For $N_0 \ge \bra{s}^{-2}$, we estimate by dividing the sign relation into $\theo = \thet$ and $\theo \neq \thet$. If $\theo = \thet$, we extract an extra time decay by using the space resonance. Precisely, simple calculation gives that
\begin{align*}
	\left|\nabla_\eta \phi_\Theta(\xi,\eta,\sigma) \right|=  \left| \frac{\xi +\eta}{ \bra{\xi+\eta}} - \frac{\xi+\eta+\sigma}{\bra{\xi+\eta+\sigma}} \right| \gtrsim \frac{L_2}{ \bra{N_1}^3}.
\end{align*}
By the integration by parts in $\eta$, $\mathcal I_{\Theta,\textbf{N},\textbf{L}}^3$ is bounded by 
\begin{subequations}
	\begin{align}
		&\begin{aligned}
		&\iint_{\mathbb{R}^{2}\times \R^2}\nabla_\eta \wt{\mathbf{m}_{\textbf{N},\textbf{L}}}(\xi,\eta,\sigma)e^{is\phi_\Theta(\xi,\eta,\sigma)} \bra{\wh{f_{\thet,N_2}}(s,\xi+\eta+\sigma), \al^\mu \wh{f_{\theth,N_3}}(s,\xi +\sigma)} \\
	&\hspace{8cm} \times \al^\nu\widehat{f_{\theo,N_1}}(s,\xi+\eta) \,d\sigma d\eta, \end{aligned}\label{eq:non-space-1}\\
&\begin{aligned}
&\iint_{\mathbb{R}^{2}\times \R^2}\wt{\mathbf{m}_{\textbf{N},\textbf{L}}}(\xi,\eta,\sigma)e^{is\phi_\Theta(\xi,\eta,\sigma)} \bra{ \wh{f_{\thet,N_2}}(s,\xi+\eta+\sigma), \al^\mu \wh{f_{\theth,N_3}}(s,\xi +\sigma)} \\
&\hspace{8cm} \times \al^\nu  \nabla_\eta \widehat{f_{\theo,N_1}}(s,\xi+\eta) \,d\sigma d\eta, \end{aligned}\label{eq:non-space-2}
	\end{align}
\end{subequations}
plus similar term, where
\begin{align}\label{eq:1-non-multiplier}
\wt{\mathbf{m}_{\textbf{N},\textbf{L}}}(\xi,\eta,\sigma) = 	\frac{\nabla_\eta \phi_\Theta(\xi,\eta,\sigma)\mathbf{m}_{\textbf{N},\textbf{L}}(\xi,\eta)}{|\nabla_\eta \phi_\Theta(\xi,\eta,\sigma)|^2}.
\end{align}
For $\bra{s}^{-2} \le N_0 \ll N_1 \sim N_2 \sim N_3 \sim L_1 \sim L_2$, Lemma \ref{lem:coifmey} yields that
\begin{align*}
	\int_{t_1}^{t_2} \sum_{\substack{\textbf{Case (i)} \\ N_0 \ge \bra{s}^{-2}}} \|\eqref{eq:non-space-1}\|_{L_\xi^2} ds	&\les \int_{t_1}^{t_2}  \sum_{\substack{\textbf{Case (i)} \\ N_0 \le \bra{s}^{-2}}}  L_1^{-1}\bra{N_1}^{13} \| \psi_{\theo,N_1} \|_{L^2} \|\psi_{\thet,N_2}\|_{L^{\infty}} \|\psi_{\theth,N_3}\|_{L^{\infty}}   ds 	\\
	&\les \int_{t_1}^{t_2}  \sum_{\substack{\textbf{Case (i)} \\ N_0 \le \bra{s}^{-2}}} L_1^{-\de}\bra{N_1}^{13-2k}   s^{-\frac32} \ve_1^3 ds
\end{align*}
and
\begin{align*}
\int_{t_1}^{t_2} \sum_{\substack{\textbf{Case (i)} \\ N_0 \le \bra{s}^{-2}}} \|\eqref{eq:non-space-2}\|_{L_\xi^2} \,ds &\les \int_{t_1}^{t_2} \sum_{\substack{\textbf{Case (i)} \\ N_0 \le \bra{s}^{-2}}} \bra{N_1}^{13}  \| x f_\theo \|_{L^2} \|\psi_{\thet,N_2}\|_{L^{\infty}} \|\psi_{\theth,N_3}\|_{L^{\infty}}  ds \\
&\les \int_{t_1}^{t_2} \sum_{\substack{\textbf{Case (i)} \\ N_0 \le \bra{s}^{-2}}} \bra{N_1}^{13-2k} \bra{s}^{-\frac32} \ve_1^3 ds.
\end{align*}
This implies  \eqref{eq:decay-non-res}  when $\theo = \thet$.

If $\theo \neq \thet$, we utilize a normal form approach by exploiting the time nonresonance
\begin{align}\label{eq:time-nonreso}
	|\phi_\Theta(\xi,\eta,\sigma) | \sim \bra{N_1}.
\end{align} 
 By the integration by parts in time, the integrand of left-hand side of \eqref{eq:decay-non-res} can be bounded by the following terms:
\begin{subequations}
	\begin{align}
		&\begin{aligned}
		&\int_{t_1}^{t_2}\iint_{\mathbb{R}^{2}\times \R^2}\frac{\mathbf{m}_{\textbf{N},\textbf{L}}(\xi,\eta)}{\phi_\Theta(\xi,\eta,\sigma)}e^{is\phi_\Theta(\xi,\eta,\sigma)} \bra{\wh{f_{\thet,N_2}}(s,\xi+\eta+\sigma), \al^\mu \wh{f_{\theth,N_3}}(s,\xi +\sigma)} \\
		&\hspace{8cm} \times \al^\nu\widehat{f_{\theo,N_1}}(s,\xi+\eta) \,d\sigma d\eta ds, \end{aligned}\label{eq:non-normal-1}\\
		&\begin{aligned}
		&\int_{t_1}^{t_2}s\iint_{\mathbb{R}^{2}\times \R^2}\frac{\mathbf{m}_{\textbf{N},\textbf{L}}(\xi,\eta)}{\phi_\Theta(\xi,\eta,\sigma)}e^{is\phi_\Theta(\xi,\eta,\sigma)} \bra{ \wh{f_{\thet,N_2}}(s,\xi+\eta+\sigma), \al^\mu \wh{f_{\theth,N_3}}(s,\xi +\sigma)} \\
		&\hspace{8cm} \times \al^\nu \partial_s \widehat{f_{\theo,N_1}}(s,\xi+\eta) \,d\sigma d\eta ds, \end{aligned}\label{eq:non-normal-2}
	\end{align}
\end{subequations}
plus similar and simpler terms (see  \eqref{eq:normal-t2}--\eqref{eq:normal-f1}). Since we have the multiplier bound
\begin{align}\label{eq:1-non-time-multiplier}
	\normo{\frac{\textbf{m}_{\textbf{N},\textbf{L}}}{\phi_\Theta}}_{\rm CM} \les \frac{N_1\bra{N_0}^5}{\bra{N_1}^3},
\end{align}
 Lemma \ref{lem:coifmey} leads us that
\begin{align*}
 \|	\eqref{eq:non-normal-1} \|_{L_\xi^2} &\les \int_{t_1}^{t_2}  \sum_{\substack{\textbf{Case (i)}\\ N_0 \ge \bra{s}^{-2}}} \frac{N_1\bra{N_0}^5}{\bra{N_1}^3} \|\psi_{\theo,N_1}\|_{L^2} \|\psi_{\thet,N_2}\|_{L^\infty}\|\psi_{\theth,N_2}\|_{L^\infty}    ds\\
 &\les \int_{t_1}^{t_2}  \bra{s}^{-\frac32+\de} \ve_1^3 ds \les \bra{t_1}^{-\frac12+\de}\ve_1^3.
\end{align*}
Using \eqref{eq:esti-timederi-f}, we also see that
\begin{align*}
	\|\eqref{eq:non-normal-2}\|_{L_\xi^2} & \les  \int_{t_1}^{t_2} s \sum_{\substack{\textbf{Case (i)}\\ N_0 \ge \bra{s}^{-2}}} \frac{N_1\bra{N_0}^5}{\bra{N_1}^3} \|\partial_sf_{\theo,N_1}\|_{L^2} \|\psi_{\thet,N_2}\|_{L^\infty}\|\psi_{\theth,N_2}\|_{L^\infty}    ds\\
	&\les  \int_{t_1}^{t_2} s \bra{s}^{-3+\de} \ve_1^3 ds \les \bra{t_1}^{-1+\de}\ve_1^3.
\end{align*}
This finishes the proof for section \ref{sec:subsub1}.

\subsubsection{Estimates for \textbf{Case (ii): $N_1 \ll N_0 \sim L_1$}.}\label{sec:subsub2} As we observed in section \ref{sec:subsub1}, \eqref{eq:condi-support}  implies that 
$$
N_1 \ll N_0 \sim N_2 \sim N_3 \sim L_1 \sim L_2.
$$
Using H\"older inequality, we see that
\begin{align*}
	\int_{t_1}^{t_2}  \normo{ \mathcal I_{\Theta,\textbf{N},\textbf{L}}^3}_{L_\xi^2} \,ds &\les \int_{t_1}^{t_2} s \bra{N_0}^{3}\|\rho_{N_0}\|_{L_\xi^2}\|\rho_{N_1}\|_{L_\eta^\frac43}        \|\wh{f_{\theo}}\|_{L_\eta^4} \|\wh{\psi_{\thet,N_2}}\|_{L_\sigma^2}\|\wh{\psi_{\theth,N_3}}\|_{L_\sigma^2}  \,ds\\
	&\les \int_{t_1}^{t_2} s \bra{N_0}^{4-2s}N_1^\frac32  \,ds.
\end{align*}
From this observation, we have the desired bound by imposing $N_1 \le \bra{s}^{-2}$. Hence, we assume that $ \bra{s}^{-2} \le N_1 $. From the fact that $N_1 \ll N_2$, we observe the inequality  
 \begin{align*}
 	|\nabla_\eta \phi_\Theta|  =\left| \frac{\xi+\eta}{\bra{\xi+\eta}} \pm \frac{\xi+\eta+\sigma}{\bra{\xi+\eta+\sigma}}\right| \gtrsim \frac{N_2}{\bra{N_2}}.
 \end{align*}
We note that this observation does not depend on the sign relation between $\theo$ and $\thet$. Therefore, we get the desired result by the integration by parts with respect to $\eta$ in a similar way to the estimates for \eqref{eq:non-space-1} and \eqref{eq:non-space-2}.

\subsubsection{Estimates for \textbf{Case (iii): $L_1 \ll N_0 \sim N_1$}.}\label{sec:subsub3} In this case, we see that 
$$
N_2 \sim N_3 \sim L_1 \ll N_0 \sim N_1 \sim L_2.
$$
This case can be handled similarly to the estimates as those in section \ref{sec:subsub2}. For this, we further divide into $L_1 \le \bra{s}^{-1}$ and $L_1 \ge \bra{s}^{-1}$.  For the first case $L_1 \le \bra{s}^{-1}$, by H\"older inequality, we have
\begin{align*}
	\sum_{\substack{\textbf{Case (iii)}\\ L_1 \le \bra{s}^{-1} }}\normo{\mathcal I_{\Theta, \textbf{N},\textbf{L}}^3(s,\xi)}_{L_\xi^2} &\les 	\sum_{\substack{\textbf{Case (iii)}\\ L_1 \le \bra{s}^{-1} }} s\bra{N_0}^5L_1 \|\rho_{L_1}\|_{L_\eta^1} \bra{N_0}^{-s}\prod_{j=1}^3 \|\psi_\thej\|_{H^s}   \\
	&\les  \bra{s}^{-2}\ve_1^3.
\end{align*}
Then, we have \eqref{eq:decay-non-res}. Let us move on to $L_1 \ge \bra{s}^{-1}$. In order to obtain the desired bound for $L_1 \ge \bra{s}^{-1}$, we also use the space resonance as follows: since $N_2 \ll N_1$, we get
 \begin{align*}
	|\nabla_\eta \phi_\Theta|  =\left| \frac{\xi+\eta}{\bra{\xi+\eta}} \pm \frac{\xi+\eta+\sigma}{\bra{\xi+\eta+\sigma}}\right| \gtrsim \frac{N_1}{\bra{N_1}},
\end{align*}
which enables us to utilize the space resonance approach as in section \ref{sec:subsub2}. Therefore, this finishes the proof of \eqref{eq:decay-non-res} for section \ref{sec:subsub3}.
 
 \subsubsection{Estimates for \textbf{Case (iv): $L_1 \sim N_0 \sim N_1$}.}
 Since all frequencies are comparable in this case, \eqref{eq:condi-support} implies
 $$
 L_2 \les N_0 \sim N_1 \sim N_2 \sim N_3 \sim L_1.
 $$
 Then we need to consider the various resonance of phase function with respect to the relation between $\thez$, $\theo$ and $\thet$. 
 
\textbf{The case $\thez=\theo = \thet$.} Let us further decompose into two cases $L_2 \ll N_0$ and $L_2 \sim N_0$. If $L_2 \ll N_0$, we have
\begin{align*}
	|\bra{\xi+\eta} - \bra{\xi+\eta+\sigma}| \les L_2 \ll \bra{N_0}.
\end{align*}
This yields the time nonresonance
\begin{align*}
	|\phi_\Theta(\xi,\eta,\sigma)| \sim \bra{N_0},
\end{align*}
which enables us to perform the normal form approach similarly to \eqref{eq:time-nonreso} to obtain the following terms:
\begin{subequations}
	\begin{align}
		&\begin{aligned}
			&\int_{t_1}^{t_2}\iint_{\mathbb{R}^{2}\times \R^2}\frac{\mathbf{m}_{\textbf{N},\textbf{L}}(\xi,\eta)}{\phi_\Theta(\xi,\eta,\sigma)}e^{is\phi_\Theta(\xi,\eta,\sigma)} \bra{\wh{f_{\thet,N_2}}(s,\xi+\eta+\sigma), \al^\mu \wh{f_{\theth,N_3}}(s,\xi +\sigma)} \\
			&\hspace{8cm} \times \al^\nu\widehat{f_{\theo,N_1}}(s,\xi+\eta) \,d\sigma d\eta ds, \end{aligned}\label{eq:4-non-normal-1}\\
		&\begin{aligned}
			&\int_{t_1}^{t_2}s\iint_{\mathbb{R}^{2}\times \R^2}\frac{\mathbf{m}_{\textbf{N},\textbf{L}}(\xi,\eta)}{\phi_\Theta(\xi,\eta,\sigma)}e^{is\phi_\Theta(\xi,\eta,\sigma)} \bra{ \wh{f_{\thet,N_2}}(s,\xi+\eta+\sigma), \al^\mu \wh{f_{\theth,N_3}}(s,\xi +\sigma)} \\
			&\hspace{8cm} \times \al^\nu \partial_s \widehat{f_{\theo,N_1}}(s,\xi+\eta) \,d\sigma d\eta ds, \end{aligned}\label{eq:4-non-normal-2}
	\end{align}
\end{subequations}
plus similar and simpler terms. As in \eqref{eq:1-non-time-multiplier}, we get
\begin{align*}
	\normo{\frac{\textbf{m}_{\textbf{N},\textbf{L}}}{\phi_\Theta}}_{\rm CM} \les N_0\bra{N_0}^2,
\end{align*}
and therefore, by Lemma \ref{lem:coifmey}, we estimate
\begin{align}
	\begin{aligned}\label{eq:4-esti-non-normal1}
	\|	\eqref{eq:4-non-normal-1} \|_{L_\xi^2} &\les \int_{t_1}^{t_2}  \sum_{\substack{\textbf{Case (iv)}\\ L_2 \le \bra{s}^{-2}}} N_0\bra{N_0}^2 \|\rho_{L_1}\|_{L_\eta^2} \|\rho_{L_2}\|_{L_\sigma^2} \prod_{j=1}^3\|\psi_{\thej,N_j}\|_{L^2} \,ds \\
	&\hspace{1cm} + \int_{t_1}^{t_2}  \sum_{\substack{\textbf{Case (iv)}\\ L_2 \ge \bra{s}^{-2}}} N_0\bra{N_0}^2 \|\psi_{\theo,N_1}\|_{L^2} \|\psi_{\thet,N_2}\|_{L^\infty}\|\psi_{\theth,N_3}\|_{L^\infty}    ds\\
	&\les \int_{t_1}^{t_2}  \bra{s}^{-\frac32+\de} \ve_1^3 ds \les \bra{t_1}^{-\frac12+\de}\ve_1^3.
	\end{aligned}
\end{align}
Using \eqref{eq:esti-timederi-f}, we also see that
\begin{align}
	\begin{aligned}\label{eq:4-esti-non-normal2}
	\|\eqref{eq:4-non-normal-2}\|_{L_\xi^2} & \les  \int_{t_1}^{t_2} s \sum_{\substack{\textbf{Case (iv)}\\ L_2 \le \bra{s}^{-2}}} N_0\bra{N_0}^2 \|\rho_{L_1}\|_{L_\eta^2} \|\rho_{L_2}\|_{L_\sigma^2} \|\partial_sf_{\theo,N_1}\|_{L^2}  \sum_{j=2}^3 \|\psi_{\thej,N_j}\|_{L^2}  \,ds\\
	&\hspace{1cm}  +\int_{t_1}^{t_2} s \sum_{\substack{\textbf{Case (iv)}\\ L_2 \ge \bra{s}^{-2}}} N_0 \bra{N_0}^2 \|\partial_sf_{\theo,N_1}\|_{L^2} \|\psi_{\thet,N_2}\|_{L^\infty}\|\psi_{\theth,N_2}\|_{L^\infty}    ds\\
	&\les  \int_{t_1}^{t_2} s \bra{s}^{-3+\de} \ve_1^3 ds \les \bra{t_1}^{-1+\de}\ve_1^3.
	\end{aligned}
\end{align}

We turn to the case $L_2 \sim N_0$. To estimate \eqref{eq:decay-non-res} for this case, we use the space resonance as follows: by integration by parts in $\eta$ with
\begin{align}\label{eq:4-non-space-resonance}
	|\nabla_\eta \phi_\Theta(\xi,\eta,\sigma)| = \left|\frac{\xi+\eta}{\bra{\xi+\eta}} - \frac{\xi+\eta+\sigma}{\bra{\xi+\eta+\sigma}} \right|\gtrsim \frac{N_0}{\bra{N_0}^3},
\end{align}
we get contribution of the form
\begin{subequations}
	\begin{align}
		&\begin{aligned}
			&\int_{t_1}^{t_2}\iint_{\mathbb{R}^{2}\times \R^2}\nabla_\eta \wt{\mathbf{m}_{\textbf{N},\textbf{L}}}(\xi,\eta,\sigma)e^{is\phi_\Theta(\xi,\eta,\sigma)} \bra{\wh{f_{\thet,N_2}}(s,\xi+\eta+\sigma), \al^\mu \wh{f_{\theth,N_3}}(s,\xi +\sigma)} \\
			&\hspace{8cm} \times \al^\nu\widehat{f_{\theo,N_1}}(s,\xi+\eta) \,d\sigma d\eta ds, \end{aligned}\label{eq:4-non-space-1}\\
		&\begin{aligned}
			&\int_{t_1}^{t_2}\iint_{\mathbb{R}^{2}\times \R^2}\wt{\mathbf{m}_{\textbf{N},\textbf{L}}}(\xi,\eta,\sigma)e^{is\phi_\Theta(\xi,\eta,\sigma)} \bra{ \wh{f_{\thet,N_2}}(s,\xi+\eta+\sigma), \al^\mu \wh{f_{\theth,N_3}}(s,\xi +\sigma)} \\
			&\hspace{8cm} \times \al^\nu  \nabla_\eta \widehat{f_{\theo,N_1}}(s,\xi+\eta) \,d\sigma d\eta ds, \end{aligned}\label{eq:4-non-space-2}
	\end{align}
\end{subequations}
and symmetric term. Here we use same  notation to \eqref{eq:1-non-multiplier}. Since all supports for frequencies are comparable, it is enough to estimate without considering the sum. Hence, we estimate
\begin{align*}
	\|\eqref{eq:4-non-space-1}\|_{L_\xi^2} \les \int_{t_1}^{t_2}  \sum_{\substack{\textbf{Case (iv)} \\L_2 \sim N_0}} L_1^{-1}\bra{N_1}^{13} \| \psi_{\theo,N_1} \|_{L^2} \|\psi_{\thet,N_2}\|_{L^{\infty}} \|\psi_{\theth,N_3}\|_{L^{\infty}}   ds 	\les \int_{t_1}^{t_2}   s^{-\frac32} \ve_1^3 ds
\end{align*}
and
\begin{align*}
	\|\eqref{eq:4-non-space-2}\|_{L_\xi^2} \les \int_{t_1}^{t_2} \sum_{\substack{\textbf{Case (iv)} \\L_2 \sim N_0}} \bra{N_1}^{13}  \| x f_\theo \|_{L^2} \|\psi_{\thet,N_2}\|_{L^{\infty}} \|\psi_{\theth,N_3}\|_{L^{\infty}}  ds \les \int_{t_1}^{t_2}  \bra{s}^{-\frac32} \ve_1^3 ds.
\end{align*}
The proof  for the case $\thez =\theo = \thet$ is done.

\textbf{The case $\thez = \theo \neq \thet$.} To handle this case, we exploit the time nonresonance phase function. Indeed, we give simple observation as follows: since $N_0 \sim N_1 \sim N_2$, we have
\begin{align*}
	\bra{\xi} - \bra{\xi+\eta} +\bra{\eta} 
	&= C\bra{N_0}^{-1} \left(-\bra{\xi}^2 + (\bra{\xi+\eta}-\bra{\eta})^2\right) \\
	&=  2C\bra{N_0}^{-1} \left[ \Big(\bra{\xi+\eta}\bra{\eta} - (|\xi+\eta||\eta|+1) \Big) + \Big(|\xi+\eta||\eta|+(\xi+\eta)\cdot\eta\Big) +\frac32 \right]\\
	& \ge 2C \bra{N_0}^{-1}\left(\frac{(|\xi+\eta|-|\eta|)^2}{\bra{\xi+\eta}\bra{\eta}+|\xi+\eta||\eta| +1} + \frac32 \right).
\end{align*}
We refer to Lemma 8.7 of \cite{canher2018} for other relations including this resonance. Then we have
\begin{align*}
	\phi_\Theta(\xi,\eta,\sigma) = \Big(-\bra{\xi} +\bra{\xi+\eta} + \bra{\eta}\Big) +\Big(- \bra{\eta}+ \bra{\xi+\eta+\sigma} + \bra{\xi+\sigma} \Big) \ge C\bra{N_0}^{-1}.
\end{align*}
Then one can perform the normal form approach to obtain  \eqref{eq:4-non-normal-1}, \eqref{eq:4-non-normal-2}, and similar  terms. The estimates for these terms agree exactly with \eqref{eq:4-esti-non-normal1} and \eqref{eq:4-esti-non-normal2}. 

\textbf{The case $\thez \neq \theo, \thez = \thet$.} This case can be estimated in a same way to the above case with the time nonresonance
\begin{align*}
	|\phi_\Theta(\xi,\eta,\sigma)| = |\bra{\xi}+\bra{\xi+\eta}+\bra{\xi+\eta+\sigma}+\bra{\xi+\sigma}| \sim \bra{N_0}.
\end{align*}

\textbf{The case $\thez \neq \theo, \thez \neq \thet$.} In order to prove desired estimates for this final case, we exploit the space resonance
\begin{align}\label{eq:4-final-multiplier}
	|\nabla_\eta \phi_\Theta(\xi,\eta,\sigma)| = \left|\frac{\xi+\eta}{\bra{\xi+\eta}} - \frac{\xi+\eta+\sigma}{\bra{\xi+\eta+\sigma}} \right| \gtrsim \frac{L_2}{\bra{N_0}^3}.
\end{align}
If $L_2 \sim N_0$, \eqref{eq:4-final-multiplier} is exactly same as \eqref{eq:4-non-space-resonance}.  Thus, we only consider $L_2 \ll N_0$. In this case, since the singularity arising from \eqref{eq:4-final-multiplier} is the most delicate and there is no time nonresonance of the phase function, we need to employ the extra null structure to give a cancellation to the singularity $L_2^{-1}$. For this, we recall $\mathcal I_{\Theta,\textbf{N},\textbf{L}}^4(s,\xi)$
\begin{align*}
	\mathcal I_{\Theta, \textbf{N},\textbf{L}}^4(s,\xi) &= s\iint_{\mathbb{R}^{2}\times \R^2}\mathbf{m}_{\textbf{N},\textbf{L}}(\xi,\eta)e^{is\phi_\Theta(\xi,\eta,\sigma)} \bra{\wh{f_{\thet,N_2}}(s,\xi+\eta+\sigma), \al^\mu \wh{f_{\theth,N_3}}(s,\xi +\sigma)}\\
	&\hspace{8.5cm}\times  \al^\nu\widehat{f_{\theo,N_1}}(s,\xi+\eta) \,d\sigma d\eta.
\end{align*}
Note that the multiplier
\[
\mathbf{m}_{\textbf{N},\textbf{L}}(\xi,\eta) =\Pi_{\thez}(\xi) \bra{\xi}^5 \bra{\eta}^{-1} \left(\thez\frac{\xi}{\langle\xi\rangle}-\theo\frac{\xi+\eta}{\langle\xi+\eta\rangle} \right) \rho_{\textbf{N},\textbf{L}}(\xi,\eta,\sigma).
\]
Then, using a simple observation
\begin{align*}
	\thez\frac{\xi}{\langle\xi\rangle}-\theo\frac{\xi+\eta}{\langle\xi+\eta\rangle} = \nabla_\xi \phi_\Theta(\xi,\eta,\sigma)  - \nabla_\sigma \phi_\Theta(\xi,\eta,\sigma),
\end{align*}
we see that
\begin{align*}
		\mathcal I_{\Theta, \textbf{N},\textbf{L}}^4(s,\xi) = 	\mathcal J_{\Theta, \textbf{N},\textbf{L}}^1(s,\xi) - 	\mathcal J_{\Theta, \textbf{N},\textbf{L}}^2(s,\xi),
\end{align*}
where
\begin{align*}
	\mathcal J_{\Theta, \textbf{N},\textbf{L}}^1(s,\xi) &= s\iint_{\mathbb{R}^{2}\times \R^2}\mathbf{m}_{\textbf{N},\textbf{L}}^1(\xi,\eta)e^{is\phi_\Theta(\xi,\eta,\sigma)} \bra{\wh{f_{\thet,N_2}}(s,\xi+\eta+\sigma), \al^\mu \wh{f_{\theth,N_3}}(s,\xi +\sigma)}\\
	&\hspace{8.5cm}\times  \al^\nu\widehat{f_{\theo,N_1}}(s,\xi+\eta) \,d\sigma d\eta,\\
	\mathcal J_{\Theta, \textbf{N},\textbf{L}}^2(s,\xi) &= s\iint_{\mathbb{R}^{2}\times \R^2}\mathbf{m}_{\textbf{N},\textbf{L}}^2(\xi,\eta)e^{is\phi_\Theta(\xi,\eta,\sigma)} \bra{\wh{f_{\thet,N_2}}(s,\xi+\eta+\sigma), \al^\mu \wh{f_{\theth,N_3}}(s,\xi +\sigma)}\\
	&\hspace{8.5cm}\times  \al^\nu\widehat{f_{\theo,N_1}}(s,\xi+\eta) \,d\sigma d\eta.
\end{align*}
Here the multipliers are defined as
\begin{align*}
	\mathbf{m}_{\textbf{N},\textbf{L}}^1(\xi,\eta,\sigma) =\Pi_{\thez}(\xi) \bra{\xi}^5 \bra{\eta}^{-1} \nabla_\xi \phi_\Theta(\xi,\eta,\sigma) \rho_{\textbf{N},\textbf{L}}(\xi,\eta,\sigma),\\
		\mathbf{m}_{\textbf{N},\textbf{L}}^2(\xi,\eta,\sigma) =\Pi_{\thez}(\xi) \bra{\xi}^5 \bra{\eta}^{-1} \nabla_\sigma \phi_\Theta(\xi,\eta,\sigma) \rho_{\textbf{N},\textbf{L}}(\xi,\eta,\sigma).
\end{align*}

\begin{rem}
	Compared to the conventional multiplier, $\nabla_\xi \phi_\Theta$ in {\rm $\textbf{m}_{\textbf{N},\textbf{L}}^1$} contains the $\sigma$-smallness from sign relation $\thez \neq \theo$, $ \thez \neq \thet$ and thus, it makes, handling the singularity arising from the space resonance, possible. Moreover,  $\nabla_\sigma \phi_\Theta$ in {\rm $\textbf{m}_{\textbf{N},\textbf{L}}^2$} gives the time decay effect as follows:
	\begin{align*}
		(\nabla_\sigma \phi_\Theta) e^{is\phi_\Theta} = \frac1s \nabla_\sigma e^{is\phi_\Theta}.
	\end{align*}
	Therefore, integration by parts in $\sigma$ leads us the desired result.	In the rest of this paper, we devote to show these approaches.
\end{rem}
By the observation \eqref{eq:4-final-multiplier}, $\mathcal J_{\Theta,\textbf{N},\textbf{L}}^1$ can be bounded as follows:
\begin{subequations}
	\begin{align}
		&\begin{aligned}
		 &\iint_{\R^2 \times \R^2}\nabla_\eta\wt{ \mathbf{m}_{\textbf{N},\textbf{L}}^1}(\xi,\eta,\sigma)e^{is\phi_\Theta(\xi,\eta,\sigma)} \bra{\wh{f_{\thet,N_2}}(s,\xi+\eta+\sigma), \al^\mu \wh{f_{\theth,N_3}}(s,\xi +\sigma)}\\
		 &\hspace{8.5cm}\times  \al^\nu\widehat{f_{\theo,N_1}}(s,\xi+\eta) \,  d\sigma d\eta, \label{eq:4-finial-space1}
		 \end{aligned}\\
		&\begin{aligned}
			&\iint_{\R^2 \times \R^2}\wt{ \mathbf{m}_{\textbf{N},\textbf{L}}^1}(\xi,\eta,\sigma)e^{is\phi_\Theta(\xi,\eta,\sigma)} \bra{\wh{f_{\thet,N_2}}(s,\xi+\eta+\sigma), \al^\mu \wh{f_{\theth,N_3}}(s,\xi +\sigma)}\\
			&\hspace{8.5cm}\times  \al^\nu \nabla_\eta\widehat{f_{\theo,N_1}}(s,\xi+\eta) \,  d\sigma d\eta, \label{eq:4-finial-space2}
		\end{aligned}
	\end{align}
\end{subequations}
and the symmetric term, where
\begin{align*}
	\wt{ \mathbf{m}_{\textbf{N},\textbf{L}}^1}(\xi,\eta,\sigma) = \frac{\nabla_\eta \phi_\Theta(\xi,\eta,\sigma)\textbf{m}_{\textbf{N},\textbf{L}}^1(\xi,\eta,\sigma)}{|\nabla_\eta \phi_\Theta(\xi,\eta,\sigma)|^2}.
\end{align*}
Since
\begin{align*}
| \nabla_\xi \phi_\Theta(\xi,\eta,\sigma)| &\les \left| \frac{\xi}{\langle\xi\rangle} +\frac{\xi+\eta}{\langle\xi+\eta\rangle} -\frac{\xi+\eta+\sigma}{\langle\xi +\eta +\sigma\rangle}-\frac{\xi+\sigma}{\langle\xi+\sigma\rangle}\right| \\
&\les \left| \frac{\xi}{\langle\xi \rangle}-\frac{\xi+\sigma}{\langle\xi+\sigma\rangle}\right| + \left| \frac{\xi+\eta}{\langle\xi + \eta  \rangle}-\frac{\xi+\eta+\sigma}{\langle\xi+\eta+\sigma\rangle}\right| \les \frac{L_2}{\bra{N_0}},
\end{align*}
we have the multiplier estimate
\begin{align*}
	\normo{\wt{ \mathbf{m}_{\textbf{N},\textbf{L}}^1}}_{\rm CM} \les \bra{N_0}^{14}.
\end{align*}
Using this, we see that
\begin{align*}
	\int_{t_1}^{t_2}\|\eqref{eq:4-finial-space1}\|_{L_\xi^2}  \,ds &\les \int_{t_1}^{t_2} \sum_{\substack{\textbf{Case (iv)} \\ L_2 \le \bra{s}^{-2}}} L_1^{-1}\bra{N_0}^{6} \|\rho_{L_1}\|_{L_\eta^2}\|\rho_{L_2}\|_{L_\sigma^2}\prod_{j=1}^{3} \|\psi_{\thej,N_j}\|_{L^2}   \,ds\\
	&\hspace{1cm} +\int_{t_1}^{t_2} \sum_{\substack{\textbf{Case (iv)} \\ L_2 \ge \bra{s}^{-2}}} L_1^{-1}\bra{N_0}^{14}\|\psi_{\theo,N_1}\|_{L^2} \|\psi_{\thet,N_2}\|_{L^\infty}\|\psi_{\theth,N_3}\|_{L^\infty}   \,ds\\
	&\les \int_{t_1}^{t_2} \bra{s}^{-2+\de} \ve_1^3 \,ds \les \bra{t_1}^{-1+\de}\ve_1^3
\end{align*}
and
\begin{align*}
	\int_{t_1}^{t_2}\|\eqref{eq:4-finial-space2}\|_{L_\xi^2}  \,ds &\les \int_{t_1}^{t_2} \sum_{\substack{\textbf{Case (iv)} \\ L_2 \le \bra{s}^{-2}}} \bra{N_0}^{6} \|\rho_{L_1}\|_{L_\eta^2}\|\rho_{L_2}\|_{L_\sigma^2} \|xf_{\theo}\|_{L^2} \prod_{j=2}^{3} \|\psi_{\thej,N_j}\|_{L^2}   \,ds\\
	&\hspace{1cm} +\int_{t_1}^{t_2} \sum_{\substack{\textbf{Case (iv)} \\ L_2 \ge \bra{s}^{-2}}} \bra{N_0}^{14}\|f_{\theo}\|_{L^2} \|\psi_{\thet,N_2}\|_{L^\infty}\|\psi_{\theth,N_3}\|_{L^\infty}   \,ds\\
	&\les \int_{t_1}^{t_2} \bra{s}^{-2+\de} \ve_1^3 \,ds \les \bra{t_1}^{-1+\de}\ve_1^{3}.
\end{align*}

Let us move on to the estimates for $\mathcal J_{\Theta,\textbf{N},\textbf{L}}^2$. Reassembling the cutoff functions again, we get  
\begin{align}
	\begin{aligned}\label{eq:4-final-sigma}
	\sum_{\textbf{N},\textbf{L}} J_{\Theta,\textbf{N},\textbf{L}}^2 &=  s \iint_{\mathbb{R}^{2}\times \R^2}\Pi_{\thez}(\xi)\bra{\xi}^{5}\bra{\eta}^{-1} \nabla_\sigma \phi_\Theta(\xi,\eta,\sigma)
	e^{is\phi_\Theta(\xi,\eta,\sigma)} \\
	&\hspace{3cm}\times  \bra{\wh{f_{\thet}}(s,\xi+\eta+\sigma), \al^\mu \wh{f_{\theth}}(s,\xi +\sigma)} \al^\nu\widehat{f_{\theo}}(s,\xi+\eta) \,d\sigma d\eta.
	\end{aligned}
\end{align}
Using integration by parts in $\sigma$ with the relation
\begin{align*}
	\nabla_\sigma \phi_\Theta(\xi,\eta,\sigma)
	e^{is\phi_\Theta(\xi,\eta,\sigma)} = \frac1s \nabla_\sigma 	e^{is\phi_\Theta(\xi,\eta,\sigma)},
\end{align*}
the right-hand side of \eqref{eq:4-final-sigma} 
is bounded by the two symmetric terms:
\begin{subequations}
	\begin{align}
	\begin{aligned}\label{eq:4-final-space3}
		& \iint_{\mathbb{R}^{2}\times \R^2}\Pi_{\thez}(\xi)\bra{\xi}^{5}\bra{\eta}^{-1}
		e^{is\phi_\Theta(\xi,\eta,\sigma)} \bra{\nabla_\sigma\wh{f_{\thet}}(s,\xi+\eta+\sigma), \al^\mu \wh{f_{\theth}}(s,\xi +\sigma)}\\
		&\hspace{8.5cm}\times  \al^\nu\widehat{f_{\theo}}(s,\xi+\eta) \,d\sigma d\eta,
	\end{aligned}\\
	\begin{aligned}\label{eq:4-final-space4}
	& \iint_{\mathbb{R}^{2}\times \R^2}\Pi_{\thez}(\xi)\bra{\xi}^{5}\bra{\eta}^{-1}
	e^{is\phi_\Theta(\xi,\eta,\sigma)} \bra{\wh{f_{\thet}}(s,\xi+\eta+\sigma), \al^\mu \nabla_\sigma \wh{f_{\theth}}(s,\xi +\sigma)}\\
	&\hspace{8.5cm}\times  \al^\nu\widehat{f_{\theo}}(s,\xi+\eta) \,d\sigma d\eta.
\end{aligned}
	\end{align}
\end{subequations}
Then we have
\begin{align*}
	\int_{t_1}^{t_2}\|\eqref{eq:4-final-space3}\|_{L_\xi^2} \,ds&\les \int_{t_1}^{t_2} \normo{\bra{D}^{-1}\bra{e^{\thet it \bra{D}}xf_\thet,\al^\mu \psi_\theth} \al^\nu  \psi_\theo    }_{H^5}     \, ds\\
	&\les \int_{t_1}^{t_2} \|xf_\thet\|_{H^5} \|\psi_\theth\|_{W^{k,\infty}} \|\psi_\theth\|_{W^{k,\infty}}   \, ds \les \bra{t_1}^{-\de} \ve_1^3
\end{align*}
and \eqref{eq:4-final-space4} can be estimated similarly. 
This completes the proof of Proposition \ref{prop:energy-weight}. \\
%
%\section{Addition}
%The (2+1) dimensional Chern-Simons
%gauge theories give a useful description to fractional quantum
%Hall effect and anyonic superconductivity.  Especially,
%the Abelian Chern-Simons theories \cite{Du, Hor} have  attracted much attention for both physicists
%and mathematicians.
% 
%    
%    The  Proca term is a classical way of give a mass to gauge fields. 
% From the mathematical point of view, the    Proca  term  plays an important role in proving a non-trivial solution to the  equations \eqref{eq:CSP-Dirac}. 
%    In fact, if  the Proca term vanish, $\lambda=0$ in \eqref{eq:CSP-Dirac}, then  we have Laplace operator $-\Delta$  for gauge fields instead of Klein-Gordon operator $-\Delta + 1$.  Then the gauge fields are represented by     
%    \[
%      A = c\int_{\Bbb{R}^2} \frac{x_j-y_j}{ |x-y|^2} J (y) dy,
%    \]
%    which gives us difficulties estimating $A_{\alpha}$. 
%   The   Proca term $\lambda \neq 0$ is important in our analysis.
%   The several mathematical papers \cite{A, Dr, H19, tsu2003} have dealt with Maxwell-Proca equations. 
%The Chern-Simons-Proca system has been studied in \cite{huh2010, Sa}.
\subsection*{Acknowledgement}
H. Huh was supported by the National Research Foundation of Korea(NRF) grant
funded by the Korea government (MSIT) (2020R1F1A1A01072197). 
K. Lee was supported in part by  the National Research Foundation of Korea(NRF) grant funded by the Korea government No. 2022R1I1A1A01056408 and No. 2019R1A5A1028324.

%%%%%%%%%%%%%%%%%%%%%%%%%%%%%%%%%%%%%%%%%%%%%%%%%%%%%%%%%%%%%%%%%%%%%%%%%%%%%%%%%%%%%%%%%%%%%%%%%%%%%%%%%%%%%%%%%%%%%%%%%%%%%%%%%%%%%%%%%%%%%%%%%%%%%%%%%%%%%%%%%%%%%%%%%%%%%%%%%%%%%%%%%%%%%%%%%%%%%%%%%%%%%%%%%%%%%%%%%%%%%%%%%%%%%%%%%%%%%%%%%%%%%%%%%%%%%%%%%%%%%%%%%%%%%%%%%%%%%%%%%%%%%%%%%%%%%%%%%%%%%%%%%%%%

\end{document}